\documentclass[11pt]{article}

\usepackage{latexsym}
\usepackage{amssymb}
\usepackage{amsthm}
\usepackage{amscd}
\usepackage{amsmath}
\usepackage{mathrsfs}
\usepackage[all]{xy}
\usepackage{hyperref} 
\usepackage[usenames,dvipsnames]{color}
\usepackage{graphicx,eepic}
\usepackage{float}
\usepackage{tabmac}
\usepackage{ytableau}

\usepackage{color}

\theoremstyle{definition}

\newtheorem* {theorem*}{Theorem}

\newtheorem{theorem}{Theorem}[section]

\theoremstyle{definition}
\newtheorem{observation}[theorem]{Observation}

\newtheorem* {example*}{Example}

\newtheorem{lemma}[theorem]{Lemma}
\theoremstyle{definition}
\newtheorem{definition}[theorem]{Definition}
\theoremstyle{definition}
\newtheorem* {notation}{Notation}
\newtheorem{conjecture}[theorem]{Conjecture}
\newtheorem{proposition}[theorem]{Proposition}
\newtheorem{corollary}[theorem]{Corollary}
\newtheorem* {remark}{Remark}
\theoremstyle{definition}
\newtheorem {example}[theorem]{Example}
\theoremstyle{definition}

\theoremstyle{definition}

\theoremstyle{definition}
\newtheorem* {remarks}{Remarks}

\xyoption{dvips}

\usepackage{fullpage}

\numberwithin{equation}{section}

\def\modu{\ (\mathrm{mod}\ }
 
\def\({\left(}
\def\){\right)}
     \newcommand{\CC}{\mathbb{C}}  
   \newcommand{\RR}{\mathbb{R}}
  \newcommand{\QQ}{\mathbb{Q}}   \newcommand{\cP}{\mathcal{P}}

\def\NN{\mathbb{N}}
\def\PP{\mathbb{P}}
    \def\ZZ{\mathbb{Z}}   \def\H{\mathcal{H}}  \def\GL{\mathrm{GL}}       \def\spanning{\textnormal{-span}}   
\def\Irr{\mathrm{Irr}}  \def\wt{\widetilde}
  
\def\diag{\mathrm{diag}}

\newcommand{\One}{{1\hspace{-.15cm} 1}}

\newcommand{\sgn}{\mathrm{sgn}}

\newcommand{\PSL}{\mathrm{PSL}}

\def\barr{\begin{array}}
\def\earr{\end{array}}
\def\ba{\begin{aligned}}
\def\ea{\end{aligned}}
\def\be{\begin{equation}}
\def\ee{\end{equation}}

\def\cF{\mathrm{F}}
\def\Uch{\mathrm{Uch}}

\makeatletter
\renewcommand{\@makefnmark}{\mbox{\textsuperscript{}}}
\makeatother

\allowdisplaybreaks[1]

\UseCrayolaColors

\begin{document}
\title{How to compute the Frobenius-Schur indicator of a  unipotent character of a finite Coxeter system}
\author{Eric Marberg
 \\ Department of Mathematics \\ Massachusetts Institute of Technology \\ \tt{emarberg@math.mit.edu}
}
\date{}

\def\ds{\displaystyle}
\def\ben{\begin{enumerate}}
\def\een{\end{enumerate}}
\def\Des{\mathrm{Des}_{\mathrm{R}}}
\newcommand{\Invol}[1]{\mathrm{Invol}(#1)}
\def\LV{\varrho_{\mathrm{LV}}}
\def\APR{\varrho_{\mathrm{APR}}}
\def\omdef{\overset{\mathrm{def}}}
\def\UIrr{\mathrm{Uch}}
\def\FT{\textbf{M}}
\def\cF{\mathcal{F}}
\def\sM{\mathscr{M}}
\def\fsd{\epsilon}
\def\FD{\mathrm{Fus}}
\def\DFD{\mathrm{DihFus}}
\def\wtFD{\mathrm{ExcFus}}

\def\qquand{\qquad\text{and}\qquad}
\def\quand{\quad\text{and}\quad}

\maketitle

\def\FakeDeg{\mathrm{FakeDeg}}
\def\Deg{\mathrm{Deg}}
\def\Eig{\mathrm{Eig}}

\setcounter{tocdepth}{2}

\begin{abstract}
For each finite, irreducible Coxeter system $(W,S)$, Lusztig has  associated a set of ``unipotent characters'' $\Uch(W)$. There is also a notion of  a ``Fourier transform'' on the space of functions $\Uch(W) \to \RR$, due to Lusztig for Weyl groups and to Brou\'e, Lusztig, and Malle in the remaining cases.
This paper concerns a certain $W$-representation $\varrho_{W}$  in the vector space generated by the  involutions of $W$. Our main result is to show that the 
irreducible multiplicities of $\varrho_W$ are given by the Fourier transform of a unique function $\fsd : \Uch(W) \to \{-1,0,1\}$, which for various reasons serves naturally as a  heuristic definition of  the Frobenius-Schur indicator on $\Uch(W)$. The formula we obtain for $\fsd$ extends prior work of Casselman, Kottwitz, Lusztig, and Vogan addressing the case in which $W$ is a Weyl group. We include in addition a succinct  description of the irreducible decomposition of $\varrho_W$ derived by Kottwitz when $(W,S)$ is classical, and prove that $\varrho_{W}$ defines a Gelfand model  if and only if $(W,S)$ has type $A_n$, $H_3$, or $I_2(m)$ with $m$ odd. We  show finally  that a conjecture of Kottwitz connecting the decomposition of $\varrho_W$ to the left cells of $W$ holds in all non-crystallographic types, and observe that a weaker form of Kottwitz's conjecture holds in general. In giving these results, we  carefully survey the construction and notable properties of the set $\Uch(W)$ and its attached Fourier transform.
 \end{abstract}

\tableofcontents

\section{Introduction}

Each finite, irreducible Coxeter system $(W,S)$ possesses a  set of ``unipotent characters'' $\Uch(W)$, introduced by Lusztig in \cite{L_app}.
When $(W,S)$ is crystallographic,
  $\Uch(W)$ arises from Lusztig's set of ``unipotent representations'' of a
  corresponding finite reductive group; Lusztig's result that this set
  depends only on $(W,S)$ (and not on the root datum) was a primary
  motivation for the definition.
  %
 By construction,
$\Uch(W)$ always contains as a subset the set $\Irr(W)$ of complex irreducible characters   of the Coxeter group $W$. However, we typically view the elements   $\Phi \in \Uch(W)$ not as characters but simply as formal objects with
 three defining attributes:
\begin{enumerate}

\item[$\bullet$] A polynomial $\FakeDeg(\Phi) \in \NN[x]$ with nonnegative integer coefficients, called the \emph{fake degree}.  

\item[$\bullet$] A nonzero polynomial $\Deg(\Phi) \in \RR[x]$ with real coefficients, called the \emph{(generic) degree}.

\item[$\bullet$] A root of unity $\Eig(\Phi) \in \CC^\times$, called the \emph{Frobenius eigenvalue}.

\end{enumerate}
It takes some care to   adequately describe $\Uch(W)$ for all finite, irreducible Coxeter systems $(W,S)$, and this description is not so well-known as that of, say, $\Irr(W)$.  Section \ref{uch-sect} below  supplies these  missing details, which for the moment we can work without.



 This paper concerns an interesting way of making sense of the question:  is there a well-defined notion of a Frobenius-Schur indicator for a ``unipotent character'' $\Phi \in \Uch(W)$? 
 Many of the results and definitions involved in telling this story can be found scattered throughout the literature, and one should thus regard this paper as  about 75\% expository$-$with the remaining 25\% dedicated in large measure to extending various results for Weyl groups to arbitrary  finite Coxeter systems.
 
Recall that if $G$ is a finite group, then the Frobenius-Schur indicator of an irreducible character $\Phi \in \Irr(G)$ is the number 
   \be\label{fsd} \fsd(\Phi) \omdef = \begin{cases} 1,&\text{if $\Phi$ is the character of a representation of $G$ in a real vector space}, \\
0,&\text{if the values of $\Phi$ are not all real}, \\ 
-1,&\text{otherwise, in which case $\Phi$ is called {quaternionic}}.
\end{cases}\ee 
Alternatively, $\fsd(\Phi)$ is also the average value of $\Phi(g^2)$ over  $g \in G$ (as is explained, for example, in the proof of \cite[Theorem 23.14]{JL}).
Since $\Uch(W)$ consists of formal objects and not characters, one cannot apply this definition directly, and it is not at all obvious what kind of definition could serve as an appropriate substitute.

If $(W,S)$ is crystallographic, there is an easy way of circumventing this difficulty.
In this case  the elements of $\Uch(W)$  are in bijection with the  unipotent characters of various finite reductive groups having $W$ as Weyl group. While not clear \emph{a priori}, it turns out that all of the actual  characters corresponding to a given
 $\Phi \in \Uch(W)$ have the same indicator$-$an observation of Lusztig \cite{LRational} which we state precisely as Proposition \ref{classical-fsd} below. 
 There is thus a logical definition of  $\fsd$ on $\Uch(W)$ when $(W,S)$ is crystallographic, which happens to have the following simple description: define $\fsd(\Phi)$ to be 1 if $\Eig(\Phi)$ is real and  $0$ otherwise.

Things become  more interesting in the case when $(W,S)$ is non-crystallographic. In this situation, to describe the Frobenius-Schur indicator in a satisfactory way, we require
 a heuristic definition 
 which is consistent with the crystallographic case but which makes sense for all types.  
A series of papers appearing in the past decade, beginning with 
 Kottwitz \cite{Kottwitz} and Casselman \cite{Casselman} and proceeding through Lusztig and Vogan \cite{LV}, suggests  such a definition, surprisingly, in terms of the irreducible multiplicities of a certain $W$-representation.
  As we will see in a moment, this approach unexpectedly leads to an extension of the Frobenius-Schur indicator to the non-crystallographic case which
which forces us to assign an indicator of $-1$ to some elements of
  $\Uch(W)$; that is, which suggests the existence of quaternionic
  ``unipotent characters.''
  
  The  $W$-representation of interest is given as follows. It is interesting to note that Adin, Postnikov, and Roichman 
 studied exactly this representation in type $A_n$ in their paper \cite{APR}. 
    
  \begin{definition}\label{intro-def} Given a finite Coxeter system $(W,S)$ with length function $\ell : W \to \NN$, let 
  \[ \Invol{W} = \QQ\spanning\{ a_w : w \in W \text{ such that }w^2=1\}\]
be a vector space with a basis indexed by the involutions in $W$, and define $\varrho_{W} : S \to \GL(\Invol{W})$ by the formula 
\[ \varrho_{W}(s) a_w 
= \begin{cases} 
-a_w,&\text{if $sw=ws$ and $\ell(ws) < \ell(w)$}, \\
a_{sws},&\text{otherwise,}
\end{cases}\qquad\text{for $s \in S$ and $w \in W$ with $w^2=1$.}\]
The map $\varrho_W$   extends to a representation of $W$ 
(a nontrivial fact, whose derivation from results of Lusztig and Vogan \cite{LV,LV2} will be explained in Section \ref{prop2.1-sect} below), and each conjugacy class of involutions in $W$
spans a $\varrho_W$-invariant subspace in $\Invol{W}$. Let $\varrho_{W,\sigma}$ denote this subrepresentation  on the space spanned by the conjugacy class of the involution $\sigma \in W$.
\end{definition}

\begin{notation} We write $\chi_W$ and $\chi_{W,\sigma}$ (when $\sigma \in W$ and $\sigma^2=1$) for the characters of $\varrho_W$ and $\varrho_{W,\sigma}$.
\end{notation}

Kottwitz \cite{Kottwitz} found a  formula for the multiplicities of the irreducible constituents of $\varrho_W$ in the case that $(W,S)$ is a Weyl group, in terms of  Lusztig's ``non-abelian Fourier transform'' \cite{L}.   (More precisely, Kottwitz computed the irreducible constituents of certain representations induced from the centralizers of involutions in $W$. The sum of these induced representations is isomorphic to $\varrho_W$, although this is not an obvious fact; see \cite[Remark 2.2]{GeckMalle_new} for a detailed explanation.)
Kottwitz proved  this formula in the classical cases, while Casselman \cite{Casselman} carried out the calculations necessary to check it in the exceptional ones. 
In more recent work, Lusztig and Vogan re-encountered  $\varrho_W$ as a specialization of a certain Hecke algebra representation, and noted a way of restating Kottwitz's results to involve the Frobenius-Schur indicator \cite[\S6.4]{LV}.  

Our main object is to describe how this last formulation extends even to the case when $(W,S)$ is non-crystallographic. 
 To this end, 
  we must briefly introduce the \emph{Fourier transform matrix} of $\Uch(W)$.  
For any finite, irreducible Coxeter system $(W,S)$, this is a real symmetric matrix $\FT$, with rows and columns  indexed by $\Uch(W)$, which possesses  the following distinguishing properties (among others; see \cite[Theorem 6.9]{GeckMalle} and also Section \ref{ft-last-sect} below): 
\begin{enumerate}
\item[(P1)] $\FT$ transforms the vector of fake degrees of $\Uch(W)$ to the vector of (generic) degrees, permuted by a certain  involution  (see Theorem \ref{p1} below).

\item[(P2)] $\FT$ is block diagonal with respect to the division of $\Uch(W)$ into families (see Section \ref{family-sect}).

\item[(P3)] $\FT$ fixes each of the vectors indexed by $\Uch(W)$ whose entries are the irreducible multiplicities, extended by zeros, of the left cell representations of $W$ (see Theorem \ref{ft-cells-thm} below).

\item[(P4)] $\FT$ and the diagonal matrix of Frobenius eigenvalues of $\Uch(W)$  determine  a representation of the modular group $\PSL_2(\ZZ)$ (see Theorem \ref{p4} below); in particular, $\FT^2=1$.


\end{enumerate}
Section \ref{FT-sect} below provides, with accompanying references, a careful description of 
the matrix $\FT$ attached to each finite, irreducible Coxeter system.
While the literature often tends to present these matrices as a somewhat heuristic construction,  work in
preparation of Brou\'e, Malle, and Michel shows that when $W$ is a primitive (complex) reflection group, $\FT$ 
 is uniquely determined
under a suitable set of natural axioms  \cite{PC}.

As a final preliminary, we  recall that  
an element $\Phi \in \Uch(W)$ is \emph{special} if the largest powers of $x$ dividing the polynomials $\FakeDeg(\Phi) \in \NN[x]$ and $\Deg(\Phi) \in \RR[x]$ are equal.
Every special $\Phi \in \Uch(W)$  belongs to the subset $\Irr(W)$. Such characters play an important role in the theory of unipotent characters for finite reductive groups and have been classified by Lusztig \cite[Chapter 4]{L}.  We mention also that Proposition \ref{special-prop} below  provides a description of the special irreducible characters of classical Coxeter systems in terms of their usual combinatorial indexing sets.

The Fourier transform $\FT$  acts on  functions $f : \Uch(W) \to \CC$ by matrix multiplication, since we may view $f$ as a  vector whose entries are indexed by $\Uch(W)$.  
The following theorem, which is our main result, extends \cite[Theorem 1]{Kottwitz} and \cite[\S6.4]{LV} to define  a function on $\Uch(W)$ which we might naturally view as the Frobenius-Schur indicator. The proof of this result appears  at the end of Section \ref{FT-sect}.

\begin{theorem}\label{main-thm} Suppose $(W,S)$ is a finite, irreducible Coxeter system with associated Fourier transform matrix $\FT$.  There then exists a \emph{unique} function $\epsilon : \Uch(W)  \to \RR$ such  that 
\begin{enumerate}
\item[(1)] $\epsilon(\Phi) \in \{-1,0,1\}$ for all $\Phi \in \Uch(W)$.

\item[(2)] $\epsilon(\Phi) = 0$ if and only if 
the Frobenius eigenvalue of $\Phi \in \Uch(W)$ is not real.

\item[(3)] 
$(\FT \epsilon)(\Phi)$ is the multiplicity of $\Phi$ in $\chi_W$ 
 for each special  $\Phi \in\Irr(W)\subset \Uch(W)$.

\end{enumerate}
For this function $\epsilon$, it  in fact holds that 
\be\label{all-note} \chi_W = \sum_{\psi \in \Irr(W)} (\FT\epsilon)(\psi)\psi.\ee
Furthermore, if  $(W,S)$ is  crystallographic, then $\epsilon$ coincides with the Frobenius-Schur indicator on $\Uch(W)$; i.e., $\epsilon(\Phi)$ is 1 or 0 according to whether $\Eig(\Phi)$ is real or non-real for $\Phi \in \Uch(W)$.


\end{theorem}

Conditions (1) and (2) are   basic properties one would desire of a prospective Frobenius-Schur indicator. In particular  we can restate (2) 
as the requirement that $\epsilon(\Phi) = 0$ if and only if $\Phi$ is not fixed by an operator on $\Uch(W)$ which we reasonably view as complex conjugation  (see Proposition \ref{cmplx-prop} below).  
The simplicity of condition (3) and the surprising fact that it implies (\ref{all-note})
make the function defined in the theorem an attractive  extension of the Frobenius-Schur indicator to the non-crystallographic case.  Admittedly, this is not  evidence that $\epsilon$ is the ``right'' extension, only that it is a suitable and interesting one.  


This Frobenius-Schur indicator gives rise to two quaternionic ``unipotent characters'' in type $H_4$. We record this phenomenon and some other  properties of $\fsd$ which become evident in the proof of Theorem \ref{main-thm} 
in the following proposition. 
\begin{proposition}\label{main-cor}
Let $\epsilon : \Uch(W) \to \{-1,0,1\}$ be the function defined in Theorem \ref{main-thm}.

\begin{enumerate}

\item[(a)] $\fsd(\Phi) = 1$ for all $\Phi \in \Irr(W)$. If $(W,S)$ is classical, then $\fsd(\Phi) =1$ for all $\Phi \in \Uch(W)$.

\item[(b)] $\epsilon(\Phi) = -1$ if and only if $(W,S)$ is of type $H_4$ and  $\Phi $ is either of  two elements of $\Uch(W)$ with \[\Deg(\Phi) = \tfrac{1}{60} x^6 + \text{higher powers of $x$}.\]


\item[(c)] $(\FT \epsilon)(\Phi)$ is a nonnegative integer for all $\Phi \in \Uch(W)$. If $(W,S)$ is  classical, then $(\FT \epsilon)(\Phi)$ is nonzero if and only if $\Phi$ is special. More generally,
  $(\FT \epsilon)(\Phi)$ is nonzero only if $\Phi \in \Irr(W)$ or  if  $\Phi$ a single element of $\Uch(W)\setminus \Irr(W)$ in each of the types $F_4$,  $E_8$, and  $H_4$.
\end{enumerate}

\end{proposition}
\begin{remark}
A lengthy computation shows that in type $H_4$,  there is no symmetric matrix $\FT$ satisfying  (P1) and (P2) for which there exists a function $\epsilon : \Uch(W) \to \{ 0,1\}$ such that (\ref{all-note}) holds.  Thus, in this sense even a different choice of Fourier transform matrix in type $H_4$ still leads to quaternionic unipotent characters.
\end{remark}


It is an interesting open problem to describe how much of the preceding theory extends from Coxeter systems to complex reflection groups. For some but not all such groups, there are analogous notions of ``unipotent characters'' and Fourier transforms (see \cite{spets}), for which  one expects  some meaningful  generalization of Theorem \ref{main-thm} to hold.

Independent of any connection to unipotent characters, the representation $\varrho_W$ 
is itself an interesting thing to study (e.g., see \cite{APR} which analyzes $\varrho_W$ in type $A_n$).
When $W$ is classical  of type $A_n$, $BC_n$, or $D_n$,   it is a natural problem to describe the irreducible decomposition of $\varrho_W$ in terms of the familiar sets of partitions $\alpha$, bipartitions $(\alpha,\beta)$, and unordered bipartitions $\{\alpha,\beta\}$ indexing $\Irr(W)$.  One can find such a description  in Kottwitz's paper \cite{Kottwitz}$-$in particular, Kottwitz shows that the set of irreducible constituents of $\chi_W$ are precisely the special characters of $W$ when $W$ is classical (note also part (c) of Proposition \ref{main-cor}). We include the following statement of these results for completeness.
(See Section \ref{class-sect} below for a precise explanation of how we label the elements of $\Irr(W)$; 
our conventions should be self-explanatory to those familiar with the subject.)


\begin{theorem}[See Kottwitz \cite{Kottwitz}]
\label{thm2}
For classical Coxeter systems $(W,S)$, the character $\chi_{W}$  
decomposes as follows:
\begin{enumerate}
\item[$\bullet$] \textbf{Type $A_n$.} 
\begin{enumerate}
\item[]$\ds\chi_{W} = \sum_\alpha \chi^\alpha,$
\end{enumerate}  where  the sum is over all partitions $\alpha$ of $n+1$.

\item[$\bullet$]  \textbf{Type $BC_n$.} \begin{enumerate}
\item[]$\ds \chi_{W}= \sum_{(\alpha,\beta)} 2^{d(\alpha,\beta)} \chi^{(\alpha,\beta)},$
\end{enumerate}  where  the sum is over all bipartitions $(\alpha,\beta)$ of $n$  such that  $\beta_i \leq \alpha_i+1$ and $\alpha_i' \leq \beta_i' + 1$ for all positive integers $i$, and where $d(\alpha,\beta)$ is the number of cells 
in the Young diagram of $\alpha\cap \beta$ which occur simultaneously at the end of a  column of $\alpha$ and   at the end of a  row of $\beta$.

\item[$\bullet$] \textbf{Type $D_n$.} 
\begin{enumerate}
\item[]$\ds\chi_{W} = \sum_{\alpha} \chi^{\{\alpha\},1} +  \sum_{\alpha} \chi^{\{\alpha\},2} +  \sum_{(\alpha,\beta)} 2^{e(\alpha,\beta)-1} \chi^{\{\alpha,\beta\}},$
\end{enumerate}
 where the first two sums are over all partitions $\alpha$ of $\frac{n}{2}$ (of which there are none if $n$ is odd), where  the third sum is over all  bipartitions $(\alpha,\beta)$ of $n$ such that $\alpha\subsetneq \beta$ and $\beta\setminus \alpha$ contains no $2\times 2$ squares, and where $e(\alpha,\beta)$ is the number of connected components of $\beta\setminus \alpha$.

\end{enumerate}

\end{theorem}

\begin{remark}
Let $W$ be the Coxeter group of type $BC_n$ and $W' \subset W$ the index two subgroup of type $D_n$. It is interesting to note that while we may view $\Invol{ W'}$ as a  $\varrho_W$-invariant subspace of $\Invol{W}$, the restriction of $\varrho_W$ on this subspace  to $ W'$ is not equivalent to the representation $\varrho_{ W'}$. 

In addition, we mention that there is some ambiguity in how the characters $\chi^{\{\alpha\},i}$ in type $D_n$ (with $n$ even) are labelled here  and in Theorem \ref{previous}. Geck's recent work \cite[Section 3.9]{Geck2} explains a satisfactory way of resolving of this ambiguity.
\end{remark}

We will actually derive this theorem from a more detailed statement, given as Theorem \ref{previous} below, describing the 
irreducible decomposition of the subrepresentations $\varrho_{W,\sigma}$ for each involution $\sigma \in W$. These results have been studied in type $A_n$ in several places previsouly;  see \cite{APR, IRS} in addition to \cite{Kottwitz}.

As a corollary to the preceding two theorems
we  note  the following result, which is interesting to compare with \cite[Theorem 1]{Vinroot}. 
Here, a \emph{Gelfand model} is a representation of a finite group which is  the multiplicity-free sum of all of the group's irreducible representations.

\begin{theorem}\label{subsub} If $(W,S)$ is a finite, irreducible Coxeter system then  
the representation $\varrho_{W}$ is a Gelfand model if and only if $(W,S)$ is of type $A_n$, $H_3$, or $I_2(m)$ with $m$ odd.
\end{theorem}

Our final results concern   a conjecture  of Kottwitz connecting the decomposition of $\chi_W$ to the left cells of $W$.  (See Section \ref{lcells} for the definition of the left cells and left cell representations of $W$.)

\begin{conjecture}[See Kottwitz \cite{Kottwitz}] \label{kottwitz-conj} 
Let $\Gamma$ be a left cell in $W$ and let $\chi_{\Gamma}$ denote the character of the corresponding left cell representation.  Let $\sigma \in W$ be an involution and write $\Sigma$ for its conjugacy class in $W$. 
Then 
$\langle \chi_{W,\sigma}, \chi_\Gamma \rangle = | \Sigma \cap \Gamma|,$ where 
$\langle\cdot,\cdot\rangle$ denotes the standard $L^2$-inner product on functions $W \to \CC$.
\end{conjecture}

Kottwitz  \cite{Kottwitz} observed that in type $A_n$ the conjecture is true, following from the known description of the left cells of $W$ in terms of the RSK-correspondence. Casselman \cite{Casselman} meanwhile verified the conjecture  in types $F_4$ and $E_6$ by a computer calculation. In Section \ref{lcells} below, we show  ourselves  that the conjecture holds  in all of the non-crystallographic cases $H_3$, $H_4$, and $I_2(m)$.  Two preprints of Geck and Bonnaf\'e \cite{Geck1,Geck2}, appearing after the completion of this article, establish several more cases, leaving the conjecture open only in type $E_8$.

We note here that a  
 weakened version of the conjecture follows   immediately from Theorem \ref{main-thm} and recent work of Geck \cite{Geck_recent}.

\begin{theorem} Let $\Gamma$ be a left cell in $W$ and let $\chi_{\Gamma}$ denote the character of the corresponding left cell representation. Then the inner product
$\langle \chi_W, \chi_\Gamma \rangle$ is equal to the number of involutions in $\Gamma$.
\end{theorem}


\begin{proof}
Proposition \ref{main-cor} and property (P3) of the Fourier transform matrix   imply that $\langle \chi_W, \chi_\Gamma \rangle = \sum_{\psi \in \Irr(W)} \langle \psi, \chi_\Gamma \rangle$, which is the cardinality of the set  $\{ \sigma \in \Gamma : \sigma^2 =1 \}$  by \cite[Theorem 1.1]{Geck_recent}.
%
\end{proof}

%
%
%
%
%

We organize this article as follows. In Section \ref{prelim}, we note several preliminaries concerning Coxeter systems, the representation $\varrho_W$, and the set $\Uch(W)$. We reproduce the calculations needed to prove Theorems \ref{thm2} and \ref{subsub} in Section \ref{cmb-sect}.  In Section \ref{FT-sect}, we describe in detail the Fourier transform matrices attached to $\Uch(W)$ and derive from this information the proof of Theorem \ref{main-thm}. Finally, in Section \ref{lcells} we prove Kottwitz's conjecture in the non-crystallographic types $H_3$, $H_4$, and $I_2(m)$. This last section explicitly describes the left cells in these types, and includes a proof of property (P3) for the Fourier transform matrix of $\Uch(W)$.

\subsection*{Acknowledgements}

I thank David Vogan and Gunter Malle for several helpful comments and suggestions.

\section{Preliminaries}\label{prelim}

Throughout, we adopt the following notational conventions:
$\NN$ is the set of nonnegative integers, $\PP$ is the set of positive integers, and $[n]$ is the set of the first $n$ positive integers, with $[0] = \varnothing$.

\subsection{Representing $W$ in $\Invol{W}$}
\label{prop2.1-sect}

Let $(W,S)$  be a finite {Coxeter system} with length function $\ell : W \to \NN$, and define the vector space $\Invol{W}$ as in the introduction.  
Here, we briefly confirm that the map $\varrho_W$ given in Definition \ref{intro-def} indeed extends to a representation of $W$. The easiest way of deriving this from known results is to prove a slightly more general fact, which goes as follows.

%
%
%
 %
 For each constant $k \in \QQ$, let $\varrho_{W,k} : S \to \GL(\Invol{W})$  be the map given by  the formula
\[ \varrho_{W,k}(s) a_w = \begin{cases} 
a_w+ka_{sw} ,&\text{if $sw=ws$ and $s \notin \Des(w)$},\\
-a_w,&\text{if $sw=ws$ and $s \in \Des(w)$}, \\
a_{sws},&\text{if $sw\neq ws$},
\end{cases}\quad\text{for $s \in S$ and $w \in W$ with $w^2=1$.}\]  Here we have written $\Des(w) = \{ s \in S : \ell(ws)<\ell(w)\}$ for the right descent set of $w \in W$, which coincides with the left descent set when $w^2=1$.
In this notation, the map $\varrho_W$ from the introduction is precisely $\varrho_{W,0}$.

Define $\H_{q^2}$ as the  Hecke algebra with parameter $q^2$ corresponding to $(W,S)$: for us, this is the  unital associative $\QQ[q]$-algebra with basis $\{ T_w : w \in W\}$ and multiplication given by the rules 
\[ \label{specialization} \begin{cases}
T_wT_{w'} = T_{ww'},&\text{for $w,w' \in W$ with }\ell(ww') = \ell(w) + \ell(w'), \\
 (T_s +1 )(T_s-q^2) =0,&\text{for }s \in S.\end{cases}\]
 In \cite{LV,LV2}, Lusztig and Vogan show that  the multiplication defined for $s \in S$ and $w \in W$ by
\[ \label{first} T_s a_w = \begin{cases} 
qa_w + (q+1)a_{sw},&\text{if $sw=ws$ and 
$s \notin \Des(w)$,
}\\
(q^2-q-1)a_w + (q^2-q)a_{sw},&\text{if $sw=ws$ and 
$s \in \Des(w)$,
}\\
a_{sws},&\text{if $sw\neq ws$ and 
$s \notin \Des(w)$,
}\\
(q^2-1)a_w + q^2 a_{sws},&\text{if $sw\neq ws$ and 
$s \in \Des(w)$.
}
\end{cases}
\]
makes the $\QQ[q]$-module
$\Invol{W}\otimes_\QQ \QQ[q]$ into an $\H_{q^2}$-module. (It is interesting  to compare this with the Hecke algebra module in \cite{APR}, which in type $A_n$ is isomorphic to the one here once we replace $q$ in \cite{APR} with $q^2$.) Specializing $q$ to 1  shows that  $\varrho_{W,2}$ is a well-defined $W$-representation, and from this  we deduce the following stronger statement.

\begin{proposition}\label{prelim-prop} Let $(W,S)$ be a finite Coxeter system. 
\begin{enumerate}
\item[(1)] The map $\varrho_{W,k}$ extends to a representation of $W$ in $\Invol{W}$ for any $k \in \QQ$.
\item[(2)] The representations $\varrho_{W,k}$ and $\varrho_{W,k'}$ are isomorphic for all $k,k' \in \QQ$.  
\end{enumerate}
\end{proposition}



\begin{proof}
Let $n(w)$ denote the dimension of the $-1$-eigenspace of $w \in W$ in its geometric representation.
Clearly $n(sws) = n(w)$ for all $w \in W$ and $s \in S$, and  it is  a straightforward exercise to show from elementary properties of the geometric representation that 
$n(sw) = n(w) + 1$ if 
$w^2=1$ and $sw=ws$ and $s \notin \Des(w)$.

List the involutions $w_1, w_2,\dots,w_N \in W$ in an order such that $i\leq j$ implies $n(w_i) \leq n(w_j)$.  
With respect to this ordered basis,
the matrices of $\varrho_{W,k}(s)$ for  $ s \in S$ are all block lower triangular and their diagonal blocks are independent of $k$. Given this observation, it follows that   $\varrho_{W,k}$  extends to a $W$-representation for all $k \in \QQ$ if and only if it does so for a single nonzero value of $k$.
As this occurs for $k=2$, part (1) follows. 
The same observation shows that  
the trace of any product of the matrices $\varrho_{W,k}(s)$ for $s \in S$  has no dependence on $k$, so $\varrho_{W,k}$ and $\varrho_{W,k'}$ have the same character, which proves (2).
\end{proof}

\subsection{References for the construction of $\Uch(W)$}\label{uch-sect}

In this  section we provide references for the construction of  $\Uch(W)$ for each finite, irreducible Coxter system $(W,S)$, 
 and also for the associated data $\FakeDeg$, $\Deg$, and $\Eig$. 
 
 Even before discussing $\Uch(W)$, we may give the general definition of the fake degree attached to a unipotent character. 
 The \emph{fake degree} of an irreducible character $\Phi \in \Irr(W)$ is the polynomial $\FakeDeg(\Phi) \in \NN[x]$ whose coefficients are the multiplicities of $\Phi$ in the graded components of the coinvariant algebra of $W$ (see \cite[\S2.4 and \S11.1]{C}). The set of irreducible characters $\Irr(W)$ always forms a subset of $\Uch(W)$, and we define the fake degrees of all $\Phi \in \Uch(W)\setminus\Irr(W)$ to be zero.
For $\Phi \in \Irr(W)$, 
$\Deg(\Phi)$ and $\Eig(\Phi)$ are defined by
\[  \Deg(\Phi) = \text{the \emph {generic degree} of $\Phi$ (see  \cite[\S10.11]{C})}
\qquad\text{and}\qquad
\Eig(\Phi) = 1,
\]
while the polynomial $\FakeDeg(\Phi)$ has the following formula \cite[Proposition 11.1.1]{C}. 

\begin{proposition}\label{fake-prop} For  $\Phi \in \Irr(W)$ we have
  \[ \FakeDeg(\Phi) = \prod_{i=1}^\ell \(1-x^{d_i}\) \cdot \frac{1}{|W|} \sum_{w \in W} \frac{\Phi(w)}{\det(1-xw)}, \]
where $d_1,\dots,d_\ell$ are the degrees of the basic polynomials invariants of $W$ (see \cite[\S2.4]{C}), and the determinant of $1-xw$ is evaluated by identifying $W$ with the image of its geometric representation.
\end{proposition}

Explicit expressions for the right-hand side of this formula when $(W,S)$ is classical appear in \cite[\S13.8]{C}.
We list the fake degrees in type $I_2(m)$  in Section \ref{i2-unipotent}. In the remaining exceptional and non-crystallographic types, one can readily compute the fake degrees  using  the
 proposition and  the character table of $W$ (see also \cite{LA,LA2,LuBe}).

Now to describe $\Uch(W)$ itself. When $(W,S)$ is crystallographic,  $\Uch(W)$ corresponds to the actual set of unipotent characters of an associated  finite reductive group.   Carter's book \cite{C} contains an excellent exposition of this correspondence; Geck and Malle's paper \cite[\S2]{GeckMalle} also serves as a useful reference.
We emphasize, however, that the construction of $\Uch(W)$ is due originally to Lusztig \cite{L,L_app}, as are the following facts.
In brief, if $G$ is a simple algebraic group defined over a finite field with $q$ elements, with Frobenius map $F : G\to G$ and Weyl group $(W,S)$, and if the finite group $G^F = \{g \in G: F(g) =g\}$ is split so that $F$ acts trivially on $W$, then the following hold:
\begin{itemize}
\item[$\bullet$] The number of unipotent characters of $G^F$, together with their 
Frobenius eigenvalues (as defined in \cite[Chapter 11]{L}), depends only on the isomorphism class of 
$(W,S)$. 
\item[$\bullet$] The  degrees of the unipotent characters of $G^F$ are given by the values at $q$ of a  set of generic degree polynomials, which also depends only on the isomorphism class of $(W,S)$.
\end{itemize}
A complete parametrization of $\Uch(W)$ in these cases, together with a list of the associated degree polynomials, appears in  \cite[\S13.8 and \S13.9]{C}, while  \cite[Theorem 11.2]{L} classifies in very simple way all of the  Frobenius eigenvalues (see also Observation \ref{eig-obs} below).
From this description derives the following proposition, which summarizes some results of Lusztig \cite{LRational} and shows that there is a natural definition of the Frobenius-Schur indicator on the formal set $\Uch(W)$ when $(W,S)$ is a Weyl group.

\begin{proposition}[See Lusztig \cite{LRational}]\label{classical-fsd} Let $G$ be a simple algebraic group over an algebraically closed field of positive characteristic, with a Frobenius map $F: G\to G$ for which the finite group $G^F$ is split. Then the irreducible unipotent characters of $G^F$ all have Frobenius-Schur indicator 1 or 0, according to whether their 
   Frobenius eigenvalues are real or non-real.
  \end{proposition}
    
  \begin{proof}
  Perhaps the simplest way to extract this result from \cite{LRational} is to 
  compare \cite[Theorem 11.2]{L} with the description of $\epsilon(\Phi)$ (from \cite{LRational}) given in \cite[\S6.4]{LV}. Alternatively, the result follows by comparing \cite[Theorem 11.2]{L} with
  \cite[Corollary 1.12]{LRational} if $G$ is classical, or with     Table 1 and Proposition 5.6 in \cite{GeckSchur}  alongside the tables in \cite[\S13.9]{C} or \cite{L} if $G$ is exceptional. 
\end{proof}


%
%
When the Coxeter system $(W,S)$ is non-crystallographic there is no corresponding reductive group, and the definition of $\Uch(W)$ is instead based  on heuristic arguments involving a list of postulates considered plausible desiderata.  
Lusztig's paper \cite{L_app} lists these postulates and constructs the set $\Uch(W)$ with the associated degree polynomials when $(W,S)$ is of any of the non-crystallographic types $H_3$, $H_4$, or $I_2(m)$.
 The corresponding Frobenius eigenvalues are given in \cite{L_app} for  $H_3$, in \cite{Malle74} for  $H_4$, and in \cite{L_exotic} for  $I_2(m)$; however,  this information is not presented in the literature as clearly as  in the crystallographic case.
Helpfully, one can access all  the relevant data in types $H_3$ and $H_4$ (as well as in the exceptional types $E_6$, $E_7$, $E_8$, $F_4$, and $G_2$)
  from the {\tt{UnipotentCharacters}} command in the computer algebra system CHEVIE \cite{CHEVIE}.
  To deal with the dihedral case, it is expedient to  review  an explicit construction of $\Uch(W)$; we do this in Section \ref{i2-unipotent} below.
  
\subsection{Two involutions $\Delta$ and $j$ of $\Uch(W)$}\label{invol-sect}
 
 We now  define two canonical involutions $\Delta$ and $j$ of the set $\Uch(W)$ which will be of importance in Section \ref{FT-sect}.
The first of these, $\Delta$, 
arises from the observation that there is  
 a  well-defined notion of formal complex conjugation on $\Uch(W)$. This fact  derives easily from the  descriptions we  cited in the previous section, and  we  record it as the following proposition for future reference.

\begin{proposition}\label{cmplx-prop} For each finite, irreducible Coxeter system $(W,S)$, there exists a \emph{unique} involutory permutation $\Delta$ of $\Uch(W)$ 
with the following properties:
\begin{enumerate}
\item[(i)] $\Delta$  preserves $\Deg(\Phi)$ and inverts $\Eig(\Phi)$ for each $\Phi \in \Uch(W)$.
\item[(ii)] $\Delta$ fixes $\Phi \in \Uch(W)$ if and only if $\Eig(\Phi) = \pm 1$.
\end{enumerate}
The involution $\Delta$ is the identity permutation if and only if $(W,S)$ is of classical type.
Furthermore, in all types, $\Delta$ fixes all elements of $\Irr(W)\subset \Uch(W)$
\end{proposition}


To define our second involution $j$, 
we recall that a polynomial $f(x)$ is \emph{palindromic} if there exists a nonnegative integer $c \in \NN$ such that $f(x) = x^c\cdot f(x^{-1})$. In particular, the zero polynomial is palindromic.
As noted
 by Opdam \cite{opdam},  the fake degree of nearly every $\Phi \in \Uch(W)$ is palindromic and  reflecting the coefficients   of $\FakeDeg(\Phi)$ about a certain central power of $x$ gives rise to a well-defined and meaningful involution of $\Uch(W)$.  (For Weyl groups $W$, this was first observed by Beynon and Lusztig \cite{LuBe}.)
To explain what we mean,  define a constant   $N_\Phi$ for each $\Phi \in \Uch(W)$ by the formula
\[ \barr{c} N_\Phi = \tfrac{1}{\Phi(1)} \sum_r \Phi(r)\text{ for $\Phi \in \Irr(W)$}
\qquad
\text{and}
\qquad
N_\Phi = 0\text{ for $\Phi \in \Uch(W)\setminus \Irr(W)$},
\earr
\] where the  sum on the left is over all reflections $r \in W$ (i.e., those $r \in W$ conjugate to elements of the generating set $S$). 
Also, let $N$ denote the total number of reflections in $W$.  
The following proposition now summarizes several observations of
Opdam \cite[Page 448]{opdam}.  (In comparing this proposition to   Opdam's paper, the reader should note  the following misprint:  the exponent of $T$ on the right-hand side of  \cite[Eq.\ (2)]{opdam} should be $N_\tau-N$ and not $N - N_\tau$.)
 
\begin{proposition}[See Opdam \cite{opdam}] \label{j-prop}
For each finite, irreducible Coxeter system $(W,S)$, there exists a \emph{unique} involutory permutationion $j$ of $\Uch(W)$ 
with the following properties:
\begin{enumerate}

\item[(i)]  The fake degree of $j(\Phi)$ is $x^{N-N_\Phi} \cdot \FakeDeg(\Phi)(x^{-1})$ for each $\Phi \in \Uch(W)$.

\item[(ii)] $j$ fixes $\Phi \in \Uch(W)$ if and only if $\FakeDeg(\Phi)$ is palindromic.

\end{enumerate}
The involution $j$ is the identity permutation if and only if $(W,S)$ is  \emph{not} of type $E_7$, $E_8$, $H_3$, or $H_4$. Furthermore, in all types, $j$ fixes all elements of $\Uch(W)\setminus \Irr(W)$.
\end{proposition}

\begin{remark} Under the assumption that $W$ is a Weyl group, this result first appeared as Proposition A in  Beynon and Lusztig's paper \cite{LuBe}.  
There is also another interpretation of the involution $j$ in terms of  rationality properties of the corresponding
   characters of  cyclotomic Hecke algebras; see \cite[Theorem 6.5]{Malle_Rationality}.
\end{remark}

Even in types $E_7$, $E_8$, $H_3$, and $H_4$, the permutation $j$ is very nearly the identity.
In particular, adopting Carter's notation for the elements of $\Irr(W)$ (see our explanations in Sections \ref{h3-sect} and \ref{h4-sect}), we may describe the nontrivial actions of  $j$  on $\Uch(W)$:   
\begin{enumerate}
\item[$\bullet$] \textbf{Type $E_7$.} $j$ exchanges  $\phi_{512,11}$ with $\phi_{512,12}$.
\item[$\bullet$]
\textbf{Type $E_8$.} $j$ exchanges
$\phi_{4096,11}$ with $\phi_{4096,12}$ and $\phi_{4096,26}$ with $\phi_{4096,27}$.

\item[$\bullet$] \textbf{Type $H_3$.} $j$ exchanges $\phi_{4,3}$ with $\phi_{4,4}$.

\item[$\bullet$] 
\textbf{Type $H_4$.} $j$ exchanges $\phi_{16,3}$ with $\phi_{16,6}$ and $\phi_{16,18}$ with $\phi_{16,21}$.

\end{enumerate}
One consequence of this description is that $\Phi \in \Irr(W)$ is special if and only if $j(\Phi \otimes \sgn)$ is special, a fact which Carter notes as \cite[Corollary 11.3.10]{C} when $W$ is crystallographic. 
Moreover, as Beynon and Lusztig  observed in \cite{LuBe}, the irreducible characters $\Phi \in \Irr(W)$ with $j(\Phi) \neq \Phi$ are precisely the characters corresponding to irreducible representations of the corresponding Hecke algebra of $W$ which are not rational. For more on this property, see also \cite[Section 11.3]{C} and \cite{L_Benson}.

\subsection{Families in $\Uch(W)$}\label{family-sect}

The set $\Uch(W)$ possesses a distinguished decomposition into disjoint subsets called \emph{families}.
When $(W,S)$ is crystallographic,
these arise as the equivalence classes 
 of a certain  relation (see \cite[Section 12.3]{C}), but in general the families are defined heuristically (see \cite{L_app}).
To specify the Fourier transform matrix of $\Uch(W)$ 
it suffices to attach  Fourier transform matrices to each of the families, 
and 
we therefore  discuss some of their significant properties here.

To begin, each family $\cF\subset \Uch(W)$  contains a unique special element $\Phi$$-$\emph{special}, we recall from the introduction,  means that there exists an integer $e \in \NN$ and nonzero real numbers $a,b$ such that 
\[ \FakeDeg(\Phi) = ax^e + \text{higher order terms}\qquad\text{and}\qquad
\Deg(\Phi) = bx^e + \text{higher order terms}.\]
 Since $\FakeDeg(\Phi) = 0$ if $\Phi \notin \Irr(W)$,  each special $\Phi$ necessarily belongs to the subset of irreducible characters $\Irr(W)\subset \Uch(W)$.  
 If $\Phi$ is the unique special element of a family $\cF$ and $e$ is defined as above, then $x^e$ is also the largest power of $x$ dividing $\Deg(\Psi)$ for all other $\Psi \in \cF$.
 Furthermore,
 if $x^{e'}$ is the largest power of $x$ dividing $\FakeDeg(\Psi)$ for some $\Psi \in \cF\cap \Irr(W)$, then $e< e'$ unless $\Phi = \Psi$.

 From this discussion we see that each family $\cF$ in $\Uch(W)$ has a nonempty intersection with $\Irr(W)$. Thus the division of $\Uch(W)$ into families induces a similar partition of $\Irr(W)$ into families.
 The family decomposition of $\Irr(W)$, in contrast to that of $\Uch(W)$, has a simple  definition in terms of the two-sided cell representations of $W$ which applies in all  types.  Namely, a family in $\Irr(W)$ consists of the characters appearing as constituents of two-sided cells having the same special constituent; see \cite[\S12.4]{C} for a detailed explanation. 

Let $\Delta$ and $j$ be the involutions of $\Uch(W)$ defined in Propositions \ref{cmplx-prop} and \ref{j-prop} above.
Every family in $\Uch(W)$ is preserved by both of these permutations, and we make the following definition concerning the action of $j$ on a family.

\begin{definition}\label{ex-def}
An element of $\Uch(W)$ is \emph{exceptional} if it is not fixed by the involution $j$, or equivalently if its fake degree is not palindromic.  (All such characters are listed at the end of Section \ref{invol-sect}.)
A family in $\Uch(W)$ is \emph{exceptional} if any of its elements are exceptional.
\end{definition}

\begin{remark}
This notion of exceptionality  originates in Beynon and Lusztig's paper \cite{LuBe} and has since become a standard definition, appearing in various places \cite{Kottwitz,opdam}. 
All exceptional families have size 4, and they only occur in types $E_7$, $E_8$, $H_3$, and $H_4$.
In types $E_7$ and $H_3$, $\Uch(W)$ contains exactly one exceptional family, while in types $E_8$ and $H_4$, $\Uch(W)$ contains two exceptional families.
\end{remark}

 

As Lusztig first observed \cite{L}, a single construction provides an extremely convenient way of parametrizing nearly every family in $\Uch(W)$.
Given a finite group $\Gamma$, let 
$C_\Gamma(x) = \{ g \in \Gamma : gxg^{-1} = x\}$ denote the centralizer of an element $x \in \Gamma$, and define $\sM(\Gamma)$ as the set of equivalence classes of pairs $(x,\sigma)$ for $x \in \Gamma$ and $\sigma \in \Irr(C_\Gamma(x))$, with respect to the relation 
\be(x,\sigma) \sim (gxg^{-1}, \sigma^g)\qquad\text{for $g \in \Gamma$,}\label{MGamma}.\ee  
Here $\sigma^g$ denotes the character of $C_\Gamma(gxg^{-1})$ with the formula $z \mapsto \sigma(g^{-1} z g)$.
Apart from one family in type $H_4$ and one family in type $I_2(m)$, each family $\cF\subset \Uch(W)$ is naturally in bijection with a set $\sM(\Gamma)$ for some 
 finite group $\Gamma$, given either by a product of 2-element cyclic groups  or a symmetric group. 
In particular, each  $\cF$ has size
\[ 1, \quad 8, \quad 21, \quad 39, \quad 74, \quad 2^{2k}, \quad k^2, \quad\text{or}\quad k^2+k+2\qquad\text{(where $k$ is a positive integer)},\]
and the families of size 74, $k^2$, and $k^2+k+2$ occur only in types $H_4$, $I_2(2k+1)$, and $I_2(2k+2)$, respectively.

If $(W,S)$ has one of the classical types $A_n$, $BC_n$, or $D_n$, then the families in $\Uch(W)$ each correspond to  sets $\sM\((\ZZ/2\ZZ)^k\)$
for various integers $k\geq 0$. If $(W,S)$ has  type $E_6$, $E_7$, $E_8$, $F_4$, or $G_2$,
 then the families in $\Uch(W)$  each correspond to $\sM(S_k)$ for some $k \in \{1,2,3,4,5\}$.
A detailed list of the families $\cF \subset \Uch(W)$ when $(W,S)$ is crystallographic, alongside the corresponding bijections $\cF \leftrightarrow \sM(\Gamma)$,  appears in \cite[\S13.8 and \S13.9]{C}.
With respect to these correspondences, the following observation is a consequence of \cite[Theorem 11.2]{L}:

\begin{observation}\label{eig-obs} Suppose $(W,S)$ crystallographic and $\cF\subset \Uch(W)$ is a family parametrized as in \cite[\S13.8 and \S13.9]{C} by the set $\sM(\Gamma)$ 
for a finite group $\Gamma$.  Let $\Phi \in \cF$ be the element corresponding to $(x,\sigma) \in \sM(\Gamma)$.
\begin{enumerate}
\item[(a)] If $\cF$ is not exceptional then $\Eig(\Phi) = \frac{\sigma(x)}{\sigma(1)}$ and $\Delta(\Phi)$ is the unique element of $\cF$ corresponding to the equivalence class $(x,\overline \sigma) \in \sM(\Gamma)$.

\item[(b)] If $\cF$ is exceptional then $\Gamma = S_2$ and $\Eig(\Phi) =1$ if $x=1$ and $\Eig(\Phi) = \sigma(x) \cdot i$ otherwise. 
In this case, $\Delta$ fixes the elements of $\cF$ with Frobenius eigenvalue 1, and exchanges the two elements with Frobenius eigenvalue $\pm i$.
\end{enumerate}
Furthermore, in either case $\Phi$ is special if and only if $(x,\sigma) = (1,\One)$, where $\One$ denotes  the principal character of $\Gamma$.
\end{observation}

Carter's book \cite{C} does not similarly address the non-crystallographic case; however,  
in these types, only a small number of families exist. A description of these families appears in Lusztig's papers papers \cite{L_exotic,L_app}, 
which we may 
 summarize as follows:
\begin{enumerate}

%
%
%
%

\item[$\bullet$] \textbf{Type $H_3$.} There are 7 families: 4 of size 1 and 3 of size 4.

\item[$\bullet$] \textbf{Type $H_4$.} There are 13 families: 6 of size 1, 6 of size 4, and 1 of size 74.

\item[$\bullet$] \textbf{Type $I_2(m)$.} There are 3 families:   they are described by (\ref{i2fam-def}) below.

\end{enumerate}
One can  find an exact parametrization of the families in 
types $H_3$ and $H_4$ (and also for the crystallographic exceptional types) in the  computer algebra system CHEVIE \cite{CHEVIE}. 
We will discuss these families at greater length in Section \ref{h-ft-sect}.

We close this subsection by considering the related combinatorial problem of counting the   elements of  $\sM(\Gamma)$. If $\Gamma$ is abelian, then  $\sM(\Gamma)=\Gamma \times \Irr(\Gamma)$ and we have $|\sM(\Gamma)| = |\Gamma|^2$.  
If $\Gamma$ is a symmetric group, then we have the following less trivial result.

\begin{proposition}\label{torics-prop}
Let $a_n = |\sM(S_n)|$ and write $\sigma(n)$ for the sum of the positive divisors of a positive integer $n$. The ordinary generating function of the sequence $\{a_n \}_{n=1}^\infty$ then satisfies
\[ 1 + \sum_{n\geq 1} a_n x^n = \prod_{n=1}^\infty \frac{1}{(1-x^n)^{\sigma(n)}}.\]  Hence, in the language of \cite{BerSlo}, $\{a_n\}_{n=1}^\infty$ is the Euler transform of the sequence $\{ \sigma(n)\}_{n=1}^\infty$.
\end{proposition}

\begin{remark}
The sequence $\{a_n\}_{n=1}^\infty =  (1, 4, 8, 21, 39, 92, 170, 360, \dots)$ appears as \cite[A061256]{OEIS}.
\end{remark}

\begin{proof}

If $C_\lambda$ denotes the centralizer in $S_n$ of a permutation with cycle type $\lambda$, then  $a_n$ is equal to sum over all partitions $\lambda$ of $n$ of the number of conjugacy classes of $C_{\lambda}$.
It is not difficult to check that if the partition $\lambda$ has $m_r$ parts of size $r$, then $C_\lambda \cong \prod_{r\geq 1} G(r,m_r)$, where $G(r,n) = (\ZZ/r \ZZ) \wr S_n$ denotes the wreath product of a cyclic group of order $r$ with $S_n$.  

The conjugacy classes of $G(r,n)$ are indexed by partitions of $n$ whose parts are labeled by numbers $i \in [r]$. 
Given a sequence $(\mu_r)_{r\geq 1}$ of such labeled partitions indexing a conjugacy class in  
$C_\lambda \cong\prod_{r\geq 1} G(r,m_r)$, 
modify the $r^{\mathrm{th}}$ partition $\mu_r$ by multiplying its parts by $r$ and replacing all labels $i$ by  ordered pairs $(r,i)$. Concatenating the parts of these modified partitions 
 yields a  map from the set of conjugacy classes 
 of $C_\lambda$ to the set  $\cP_n$ of 
  partitions  of $n$ whose parts of size $k$ are labeled by pairs $(d,i)$ where $d$ is a divisor of $k$ and $i \in [d]$.  
  Conversely,  any  $\nu \in \cP_n$ determines a sequence  of labeled partitions 
 $(\mu_r)_{r \geq 1}$ which indexes a conjugacy class of some $C_\lambda$: namely, take $\mu_r$ to be the labeled partition given by  dividing by $r$ 
  all parts of $\nu$ which are labeled by pairs of the form $(r,i)$.
  These operations  determine a bijection from $\cP_n$ to  the disjoint union  of the sets of conjugacy classes of $C_\lambda$ over all partitions $\lambda$ of $n$, and thus $a_n  = |\cP_n|$. 
 
 We can view $\cP_n$ as the set of 
 partitions of $n$ whose parts of size $k$ can have $\sigma(k)$ different types.
The cardinality of this set is   what is counted by 
the coefficient of $x^n$ in the right-hand generating function above, essentially by construction
(see \cite{BerSlo}), which completes the proof.
\end{proof}

  \subsection{Unipotent characters in type $I_2(m)$}
  \label{i2-unipotent}

It seems difficult to find a single reference which provides all of the  data  our setup requires  to construct $\Uch(W)$ in type $I_2(m)$. We therefore briefly include the relevant details here, with   
the papers \cite{L_exotic,L_app,Malle_Imprimitiver} serving as our primary references.

 Fix $m\geq 3$ and let $(W,S)$ be of type $I_2(m)$.   The set $\Uch(W)$ then consists of the objects   \begin{enumerate}
  \item[] $\Phi_{(0,j)}$ for  integers $j$ with $0<j < \frac{m}{2}$,
  \item[]  $\Phi_{(i,j)}$ for  integers $i,j$ with $0 < i<j<i+j < m$,
  \end{enumerate}
  together with the additional objects \[\begin{cases} \text{$\One$, $\sgn$},&\text{if $m$ is odd} \\
\text{$\One$, $\sgn$, $\Phi_{(0,\frac{m}{2})}'$, $\Phi_{(0,\frac{m}{2})}''$},&\text{if $m$ is even.}\end{cases}\]
Observe that $\Uch(W)$ has cardinality  $k^2+2$  if $m = 2k+1$ is odd or $k^2-k+4$ if $m=2k$ is even.

 Let $\xi = \exp\(\frac{2\pi \sqrt{-1}}{m}\)$ be the standard primitive $m^{\mathrm{th}}$ root of unity.
The fake degrees and degrees for $\Uch(W)$ are then defined by
\begin{enumerate}

\item[] $\FakeDeg\(\Phi_{(i,j)}\) = \begin{cases} x^j + x^{m-j},&\text{if }i=0, \\ 0,&\text{otherwise},\end{cases}$



\item[] $\FakeDeg\(\Phi_{(0,\frac{m}{2})}'\) = \FakeDeg\(\Phi_{(0,\frac{m}{2})}''\) = x^{m/2}$,

\end{enumerate}
and
\begin{enumerate}

\item[] $\Deg\(\Phi_{(i,j)}\) = \frac{\xi^i + \xi^{m-i} - \xi^j - \xi^{m-j}}{m} \cdot x \cdot \frac{(x-1)(x+1) \prod_{k\in [m]} (x-\xi^k)}{(x-\xi^i)(x-\xi^{m-i})(x-\xi^j)(x-\xi^{m-j})}$,
 

\item[] $\Deg\(\Phi_{(0,\frac{m}{2})}'\) = \Deg\(\Phi_{(0,\frac{m}{2})}''\) = \frac{2}{m} \cdot x \cdot \prod_{0<k<\frac{m}{2}} (x-\xi^k)(x-\xi^{m-k})$,

\end{enumerate}
with  
$\FakeDeg\(\One\)=\Deg(\One) = 1$ and  $\FakeDeg\(\sgn\)=\Deg(\sgn) = x^m$. Note that in the right-hand expression for $\Deg\(\Phi_{(i,j)}\)$, the denominator of the last factor always divides its numerator, and so the degree does belong to $\RR[x]$.
The Frobenius eigenvalues have the formula
\[ \Eig(\Phi) = \begin{cases} \xi^{-ij}, &\text{if $\Phi = \Phi_{(i,j)}$ for some $(i,j)$}, \\ 1,&\text{otherwise}.\end{cases}\]
We note that in this type,  the  permutation $\Delta$ of $\Uch(W)$ defined in Proposition \ref{cmplx-prop}
fixes all elements except those of the form $\Phi_{(i,j)}$ with $i>0$, on which it acts by  
$\Delta : \Phi_{(i,j)} \mapsto \Phi_{(i,m-j)}$.

Finally,   $\Uch(W)$ always has exactly three families, given by the sets
\be\label{i2fam-def}  \{\One\},\qquad \{\sgn\},\qquad \text{and}\qquad \Uch(W)\setminus \{\One,\sgn\}.\ee
The special elements of $\Uch(W)$ are then $\One$, $\sgn$, and $\Phi_{(0,1)}$. 

\section{Decomposing $\chi_{W}$ for arbitrary Coxeter systems} \label{cmb-sect}

Kottwitz and Casselman's papers \cite{Kottwitz,Casselman} give the decomposition of $\chi_W$ when $W$ is a  Weyl group. Here we   explain the derivation of one  reformulation of Kottwitz's results for classical Coxeter systems, then describe the irreducible decomposition in the remaining non-crystallographic types.

\subsection{Kottwitz's results for classical types}\label{class-sect}

We devote this section to the derivation of Theorem \ref{thm2} in the introduction.
To begin, we recall  the following explicit constructions of  the Coxeter systems of classical type:
\begin{enumerate}

\item[$\bullet$]  \textbf{Type $A_n$.} $W$ is the symmetric group $S_{n+1}$ of permutations of $[n+1]$ and $S = \{s_1,\dots,s_n\}$ where $s_i$ is the simple transposition $(i,i+1) \in S_{n+1}$ for $i\in [n]$.

\item[$\bullet$]  \textbf{Type $BC_n$.}  $W $ is the  group of $n\times n$ generalized permutation matrices with entries in $\{-1,0,1\}$, a group  we occasionally call
$W_n$. The  set $S = \{s_1,\dots,s_{n-1},t_n\}$ consists of the permutations  $s_i \in S_n$ (viewed as matrices)  together with the  matrix $t_n = \diag(1,1,\dots,1,-1)$.

\item[$\bullet$]  \textbf{Type $D_n$.} $W$ is the subgroup of matrices in the Coxeter group of type $BC_n$ with an even number of entries equal to $-1$. The generating set  $S = \{s_1,\dots,s_{n-1},t_n'\}$ consists of the permutations  $s_i \in S_n$ (viewed as matrices)  together with the matrix $t_n' = t_n s_{n-1} t_n$.
\end{enumerate}

To refer to the irreducible representations of these groups we employ the following conventions.  First,  a \emph{partition} of an integer $n$ is a weakly decreasing sequence of  positive integers $\alpha= (\alpha_1,\alpha_2,\dots,\alpha_k)$ which sum to $n$.  For convenience we define $\alpha_i$ to be $0$ for all $i \in \PP$ exceeding $k$ and write $|\alpha|$ for the sum $\sum_{i \in \PP} \alpha_i$.
The \emph{Young diagram} of a partition $\alpha$ is the subset of $\PP\times \PP$ given by $\{ (i,j) : 1 \leq j \leq \alpha_i\}$, which we  draw in ``English notation'' as in the following example: \[\alpha = (4,2,1)\qquad\text{has Young diagram}\qquad \tableau[s]{&&& \\ & \\ \ }.\]
We  often identify a partition $\alpha$ with its Young diagram; for example, we write  $\alpha \subset \beta$ when the Young diagram of $\alpha$ is contained in that of $\beta$, and write $ \alpha\cap\beta$ for the partition whose $i^{\mathrm{th}}$ part is $\min \{ \alpha_i,\beta_i\}$. 
We denote the \emph{transpose} of $\alpha$ by $\alpha'$; recall that $\alpha'_i = |\{ j \in \PP : \alpha_j \leq i\}|$, and that consequently the Young diagram of $\alpha'$ is the transpose of the  diagram of $\alpha$.

All Coxeter groups $(W,S)$ have two linear characters given by the trivial character $\One : w \mapsto 1$ and the sign character $\sgn : w \mapsto (-1)^{\ell(w)}$. The remaining irreducible characters of the the classical Coxeter groups may be described as follows:

\begin{enumerate}
\item[$\bullet$]  \textbf{Type $A_n$.}  For each partition $\alpha $ of $ n+1$, there exists a unique irreducible character $\chi^\alpha$ of $W = S_{n+1}$
whose restrictions to the subgroups $S_{\alpha_1} \times S_{\alpha_2} \times \cdots$ and $S_{\alpha'_1} \times S_{\alpha'_2} \times \cdots$
 respectively contain the trivial character and sign character as constituents.  The characters $\chi^\alpha$ for partitions $\alpha $ of $ n+1$ are distinct and exhaust $\Irr(W)$.

\item[$\bullet$]  \textbf{Type $BC_n$.}  
%
%
If $\alpha$ is a partition then
the irreducible character $\chi^\alpha$ of $S_{|\alpha|}$  extends in two ways to characters $\chi^\alpha_+$ and $\chi^\alpha_-$ of $W_{|\alpha|}$ by the  formulas
\[\label{irr-b} \chi^\alpha_+(g) = \chi^\alpha(|g|)\qquad\text{and}\qquad 
\chi^\alpha_-(g) = \Delta(g) \chi^\alpha(|g|),\qquad\text{for }g \in W_{|\alpha|},\] where $|g|$ denotes the permutation matrix formed by replacing every entry of $g$ with its absolute value, and $\Delta(g)$ is the product of the nonzero entries of $g$.
%
%
Given an ordered pair of partitions $(\alpha,\beta)$ with $|\alpha|+|\beta| = n$ (i.e., a \emph{bipartition} of $n$),  define $\chi^{(\alpha,\beta)}$ as the  character of $W=W_n$ induced
from the character $\chi^\alpha_+ \otimes \chi^\beta_-$ of the subgroup $W_{|\alpha|}\times W_{|\beta|}$.
The characters $\chi^{(\alpha,\beta)}$ for bipartitions $(\alpha,\beta)$ of $n$ are distinct and irreducible, and exhaust  $\Irr(W)$.

\item[$\bullet$]  \textbf{Type $D_n$.} Define an \emph{unordered bipartiton} of $n$ to be a set of two \emph{distinct} partitions $\{\alpha,\beta\}$ with $|\alpha|+|\beta| = n$. If $\{\alpha,\beta\}$ is an unordered bipartition of $n$ then 
the  irreducible characters $\chi^{(\alpha,\beta)}$ and $\chi^{(\beta,\alpha)}$ of $W_n$ have the same, irreducible restriction to $W$, which we denote by $\chi^{\{ \alpha,\beta\}}$. 
If $n$ is even and $\alpha$ is a partition of $n/2$ then the restriction of $\chi^{(\alpha,\alpha)}$ to $W$ has two irreducible constituents, which we denote in no particular order by $\chi^{\{\alpha\},1}$ and $\chi^{\{\alpha\},2}$.
The characters $\chi^{\{\alpha,\beta\}}$ for unordered partitions $\{\alpha,\beta\}$ of $n$ together with  $\chi^{\{\alpha\},1}$ and $\chi^{\{\alpha\},2}$ for partitions $\alpha$ of $n/2$ exhaust the distinct elements of $\Irr(W)$.

\end{enumerate}

\begin{remark}
 Lusztig \cite[Chapter 4]{L} describes an alternate indexing set for the irreducible characters of the classical Coxeter groups,  based on equivalence classes of certain two-line arrays. This notation 
is  fundamentally important in a variety of contexts and provides simple conditions for deciding when two characters in $\Irr(W)$  belong to the same family in $\Uch(W)$ or when a character a special.   
\end{remark}

The following proposition describes the special characters of $\Irr(W)$ in terms of the indexing sets just introduced. 
This description is similar to Spaltenstein's characterizations of the special representations of the classical Weyl groups in \cite[Section 4]{Spalt}.
(For the definition of the families in $\Irr(W)$, see \cite[\S13.2]{C}.)

\begin{proposition} \label{special-prop} The special irreducible characters of the finite Coxeter systems $(W,S)$ of classical type are described as follows:
\begin{enumerate}
\item[$\bullet$]  \textbf{Type $A_n$.}  $\chi^\alpha $ is special for all partitions $\alpha$ of $n+1$.

\item[$\bullet$]  \textbf{Type $BC_n$.}  $\chi^{(\alpha,\beta)} $ is special if and only if $\beta_i \leq \alpha_i + 1 $ and $\alpha_i' \leq \beta'_i + 1$ for all $i \in \PP$.

\item[$\bullet$]  \textbf{Type $D_n$.} $\chi^{\{\alpha,\beta\}}$ is special if and only if  one of the partitions $\alpha$, $\beta$ properly contains the other and the corresponding skew diagram contains no $2\times 2$ squares. Both  $\chi^{\{\alpha\},1}$ and $\chi^{\{\alpha\},2}$ are special for all partitions $\alpha$ of $n/2$ when $n$ is even.
\end{enumerate}
\end{proposition}

\begin{remarks} Before giving its proof, let us explain the proposition pictorially.\begin{enumerate}
\item[(1)] The condition for $\chi^{(\alpha,\beta)}$ to be special in type $BC_n$ corresponds to the following picture: 
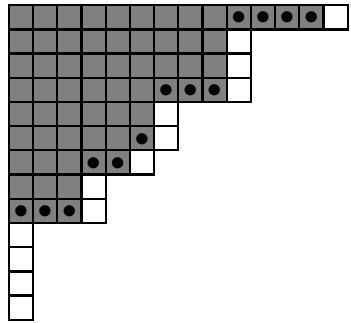
\[
\ytableausetup 
{boxsize=0.8em} 
\begin{ytableau}
*(gray) & *(gray) & *(gray) &*(gray) &*(gray) &*(gray) &*(gray) & *(gray) &  *(gray) &*(gray) \bullet&*(gray)  \bullet &  *(gray)  \bullet &*(gray)  \bullet &  \\
*(gray) &*(gray) &*(gray) &*(gray) &*(gray) &*(gray) &*(gray) &*(gray) &*(gray) &   \\
*(gray) &*(gray) &*(gray) &*(gray) &*(gray) &*(gray) &*(gray) &*(gray) &*(gray) &   \\
*(gray) &*(gray) &*(gray) &*(gray) &*(gray) &*(gray) &*(gray)\bullet &*(gray)\bullet &*(gray) \bullet&   \\
*(gray) &*(gray) &*(gray) &*(gray) &*(gray) &*(gray) &   \\
*(gray) &*(gray) &*(gray) &*(gray) &*(gray) &*(gray)\bullet &   \\
*(gray) &*(gray) &*(gray) &*(gray)\bullet &*(gray) \bullet&   \\
*(gray) &*(gray) &*(gray) &    \\
*(gray)\bullet &*(gray)\bullet &*(gray) \bullet&    \\
  \\
  \\  
  \\ \\
   \end{ytableau}
\] If the Young diagram of $\alpha$ is the set of gray cells, then $\chi^{(\alpha,\beta)}$ is special if and only if $\beta$ is formed from $\alpha$ by adding (a subset of) white cells and/or removing (a subset of) gray cells marked by  $\bullet$'s.

\item[(2)]  The condition for $\chi^{\{\alpha,\beta\}}$ to be special in type $D_n$ may be alternately stated  as the requirement that $\alpha\subsetneq\beta$ and $\beta_{i+1} \leq \alpha_i+1$ for all $i \in \PP$ (or that the same condition hold with $\alpha$ and $\beta$ reversed).
The relevant picture here is the following:
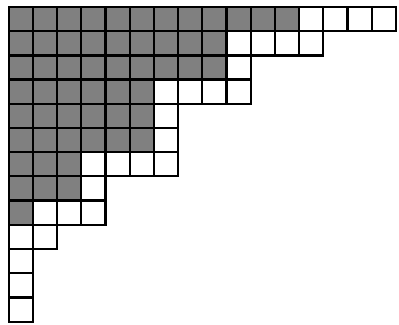
\[
\ytableausetup 
{boxsize=0.8em} 
\begin{ytableau}
*(gray) & *(gray) & *(gray) &*(gray) &*(gray) &*(gray) &*(gray) & *(gray) &  *(gray) &*(gray) &*(gray)   &  *(gray)   & & & &    \\
*(gray) &*(gray) &*(gray) &*(gray) &*(gray) &*(gray) &*(gray) &*(gray) &*(gray) & & & &   \\
*(gray) &*(gray) &*(gray) &*(gray) &*(gray) &*(gray) &*(gray) &*(gray) &*(gray) &   \\
*(gray) &*(gray) &*(gray) &*(gray) &*(gray) &*(gray) &  &  & &   \\
*(gray) &*(gray) &*(gray) &*(gray) &*(gray) &*(gray) &   \\
*(gray) &*(gray) &*(gray) &*(gray) &*(gray) &*(gray) &   \\
*(gray) &*(gray) &*(gray) & & & &   \\
*(gray) &*(gray) &*(gray) &    \\
*(gray)  & &&    \\
 & \\
  \\  
  \\ \\
   \end{ytableau}
\] If the Young diagram of $\alpha$ is the set of gray cells and $\alpha\subset\beta$, then $\chi^{\{\alpha,\beta\}}$ is special if and only if $\beta$ is formed by adding (a nonempty subset of) white cells to $\alpha$.  Note that our diagram truncates the set of  white cells shown in the first row and first column.
\end{enumerate}
\end{remarks}

\begin{proof}[Proof of Proposition \ref{special-prop}]
That every irreducible character in type $A_n$ is special is precisely \cite[Proposition 11.4.1]{C}.
The  condition in type $BC_n$ for a character $\chi^{(\alpha,\beta)}$ to be special \cite[Propositions 11.4.2 and 11.4.3]{C} is usually stated  as the requirement that 
$\lambda_1 \leq \mu_1 \leq \lambda_2 \leq \mu_2 \leq \dots \leq \mu_m \leq \lambda_{m+1}$ where 
\be\label{b-special}\lambda_i = \alpha_{m+1-(i-1)} + (i-1) \qquad\text{and}\qquad \mu_i = \beta_{m-(i-1)} + (i-1),\ee 
and where $m$ is any  integer large enough so that $m+1$ and $m$ are at least the number of parts of $\alpha$ and $\beta$, respectively.  In particular, one can  take $m=n$ since $|\alpha| + |\beta| = n$.  It is simply a matter of rewriting this statement to deduce that $\chi^{(\alpha,\beta)}$ is special if and only if $ \alpha_{i+1} \leq \beta_i \leq \alpha_{i}+1$ for all $ i \in \PP$. In turn, it is straightforward to check that $\alpha_{i+1} \leq \beta_i$ for all $i \in \PP$ if and only if $\alpha_i' \leq \beta_i'+1$ for all $ i \in \PP$.

Similarly, in type $D_n$ the character $\chi^{\{\alpha,\beta\}}$ is special if and only if 
$ \lambda_1 \leq \mu_1  \leq \dots \leq \lambda_m \leq  \mu_m
$ or $
 \mu_1 \leq \lambda_1 \leq \dots \leq \mu_m \leq  \lambda_m$
 where \be\label{d-special}\lambda_i = \alpha_{m-(i-1)}+(i-1) \qquad\text{and}\qquad \mu_i = \beta_{m-(i-1)} + (i-1)\ee and $m$ is any sufficiently large integer \cite[Proposition 11.4.4]{C}.  This is equivalent to the requirement that $\alpha_i \leq \beta_i$ and $\beta_{i+1} \leq \alpha_i+1$ for all $i \in \PP$ (or the same condition with $\alpha$ and $\beta$ reversed), and it is straightforward to check that this holds if and only if $\alpha\subsetneq \beta$ and $\beta\setminus \alpha$ contains no $2\times 2$ squares.  Finally, Carter \cite[Proposition 11.4.4]{C} notes that $\chi^{\{\alpha\},1}$ and $\chi^{\{\alpha\},2}$ are special in type $D_{2n}$ for all partitions $\alpha$ of $n$.
\end{proof}

The involutions of $W$  are  the permutations whose cycles all have length one or two in type $A_n$, and   the  matrices which are symmetric in type $BC_n$ or $D_n$. 
For each nonnegative integer $m$, let $ \sigma_m$ denote the permutation $(1,m+1)(2,m+2)\cdots(m,2m)$ or (when appropriate) the  corresponding $2m\times 2m$ permutation matrix.  The permutations $ \sigma_m$ for $0 \leq m \leq \lfloor \frac{n+1}{2} \rfloor$ represent the distinct conjugacy classes of involutions in the Coxeter group of type $A_n$, while in type $BC_n$,  the matrices 
\[  \sigma_{k,\ell,m} = \(\barr{ccc}  \sigma_m & 0 & 0\\ 0 & I_k & 0 \\ 0 & 0 & -I_\ell \earr\),\quad\text{for integers $k,\ell,m\geq 0$ with $2m+k+\ell = n$}\]
provide the desired set of representative involutions.   In the type $D_n$, the matrices $ \sigma_{k,\ell,m}$ with 
$2m+k+\ell = n$ and $\ell$ even
each belong to a distinct conjugacy class of involutions.  If $n$ is odd then
any involution is conjugate to one of these elements; if $n$ is even, then there is one additional conjugacy class of involutions represented by  the product of $ \sigma_{0,0,m/2}$ and the diagonal matrix $\diag(1,\dots,1,-1,-1)$, an element we denote $\sigma'_{0,0,m/2}$.

As the main result of this section we now describe how the characters $\chi_{W,\sigma}$  for   involutions $\sigma \in W$ decompose when $(W,S)$ is of classical type.  Kottwitz  derives equivalent forms of  these decompositions in his paper \cite{Kottwitz}, and our primary contribution is to  collect his findings in one concise statement.

\begin{theorem}[See Kottwitz \cite{Kottwitz}]
\label{previous}
For a finite Coxeter system $(W,S)$ of classical type, the characters $\chi_{W,\sigma}$ 
for  involutions $\sigma \in W$ 
decompose as follows:
\begin{enumerate}
\item[$\bullet$] \textbf{Type $A_n$.} Let $m$ be a nonnegative integer with $2m\leq n+1$.  Then
\begin{enumerate}
\item[] $\ds\chi_{W, \sigma_m} = \sum_\alpha \chi^\alpha,$  
\end{enumerate}
where  the sum is over all partitions $\alpha$ of $n+1$ with exactly $n+1-2m$ odd columns.

\item[$\bullet$]  \textbf{Type $BC_n$.} Let $k,\ell,m\geq 0$ be  integers with $2m+k+\ell = n$.  Then 
\begin{enumerate}
\item[] $\ds \chi_{W, \sigma_{k,\ell,m}} = \sum_{(\alpha,\beta)} \binom{d(\alpha,\beta)}{|\alpha\cap \beta| - m} \chi^{(\alpha,\beta)},\quad$
\end{enumerate}
where  the sum is over  bipartitions $(\alpha,\beta)$ with $|\alpha| = k+m$ and $|\beta| = \ell+m$ such that  $\beta_i \leq \alpha_i+1$ and $\alpha_i' \leq \beta_i' + 1$ for all $i \in \PP$, and where
\begin{enumerate}
\item[] 
$d(\alpha,\beta)$ is the number of pairs $(i,j) \in \PP\times \PP$ for which $i=\alpha'_j$ and $j=\beta_i$.
\end{enumerate}

\item[$\bullet$] \textbf{Type $D_n$.} Let $k,\ell,m \geq 0$ be integers with $2m+k+\ell =n$ and $\ell$ even and $k+\ell\neq 0$. Then
\begin{enumerate}
\item[] $\ds \chi_{W, \sigma_{k,\ell,m}} = \sum_{(\alpha,\beta)} \binom{e(\alpha,\beta)}{f(\alpha,\beta)-\ell} \chi^{\{\alpha,\beta\}},
\quad$
\end{enumerate}
 where the  sum is over all  bipartitions $(\alpha,\beta)$ with $|\alpha| = m$ and $|\beta| = k+\ell+m$ such that $\alpha\subset \beta$ and $\beta\setminus \alpha$ contains no $2\times 2$ squares, and where
 \begin{enumerate}
\item[] $e(\alpha,\beta)$ is the number of connected components of the skew diagram $\beta\setminus\alpha$,
\item[] $f(\alpha,\beta)$ is the number of positive integers $i$ with  $\alpha_i \neq \beta_i$.
\end{enumerate}

In addition, if $n$ is even then one can label the  characters $\chi^{\{\alpha\},1},\chi^{\{\alpha\},2} \in \Irr(W)$ such that 
\begin{enumerate}
\item[] $\ds\chi_{W, \sigma_{0,0,n/2}} = \sum_{\alpha} \chi^{\{\alpha\},1}$\ \ and \ \
$\ds\chi_{W, \sigma_{0,0,n/2}'} = \sum_{\alpha} \chi^{\{\alpha\},2},$
\end{enumerate}
where both sums are over all partitions $\alpha$ of $n/2$.
\end{enumerate}

\end{theorem}

\begin{remark}
Theorem \ref{thm2} in the introduction follows by summing each of these decompositions over the corresponding  representative set of involutions, noting that we evaluate the binomial coefficient $\binom{n}{k}$ to be zero if $k>n$ or $k<0$. 
%
%
%

Concerning the decomposition in type $D_n$, we recall that the connected components of a skew diagram are the equivalences classes of cells under the relation of (left-right or up-down) adjacency. For example, if $\beta = (7,5,2,2)$ and $\alpha = (5,4,1)$ then the skew diagram $\beta\setminus \alpha$, shown as the white cells in the picture 
 \[
\ytableausetup{boxsize=1em} 
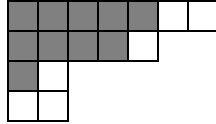
\begin{ytableau} 
*(gray) &*(gray) &*(gray) &*(gray) &*(gray) &  &  \\ 
*(gray) &*(gray) &*(gray) &*(gray) & \\ 
*(gray) & \\ 
 & 
\end{ytableau}
 \]
has three connected components.
\end{remark}

\begin{proof}
The decomposition of $\chi_{W,\sigma}$ in type $A_n$ appears as \cite[Corollary A.1]{APR} or  \cite[Lemma 2]{IRS} or \cite[\S3.1]{Kottwitz}.

In type $BC_n$, given a bipartition $(\alpha,\beta)$ of $n$, Kottwitz \cite[\S3.2]{Kottwitz} proves that the multiplicity of $\chi^{(\alpha,\beta)}$ in $\chi_{W,\omega_{k,\ell,m}}$ 
is zero unless 
$\chi^{(\alpha,\beta)}$  
is special and $|\alpha| =k+m$ and $|\beta| = \ell+m$, in which case the  multiplicity  
is 
$\binom{d}{|\alpha\cap\beta|-m}$, where 
\begin{enumerate}
\item[$\bullet$] $d$ is the number of integers in the sequence 
$\mu_1,\mu_2,\dots,\mu_n$ which do not occur in the sequence $\lambda_1,\lambda_2,\dots,\lambda_{n+1}$, where  $\lambda_i = \alpha_{n+1-(i-1)} + (i-1)$ and $\mu_i = \beta_{n-(i-1)} + (i-1)$ as in (\ref{b-special}).
\end{enumerate}
Given Proposition \ref{special-prop}, we 
have only to show that when $\chi^{(\alpha,\beta)}$ is special, $d$ is equal to the number $d(\alpha,\beta)$ defined in the theorem. 
As noted in the proof of Proposition \ref{special-prop}, if $\chi^{(\alpha,\beta)}$ is special then $\lambda_1\leq \mu_1 \leq \dots \leq \lambda_n \leq  \mu_n \leq \lambda_{n+1}$, and so $d$ is equal to the number of $i \in [n]$ with $\lambda_i < \mu_i <\lambda_{i+1}$, which by definition is the number of $i \in \PP$ with $\alpha_{i+1} < \beta_i \leq \alpha_i$. This last formulation is precisely the number of pairs $(i,j) \in \PP\times \PP$ for which $i=\alpha_j'$ and $j=\beta_i$, as desired.

Now assume $(W,S)$ is of type $D_n$.
Kottwitz notes the given decompositions of $\chi_{W, \sigma_{0,0,n/2}}$ and $\chi_{W, \sigma_{0,0,n/2}'}$ when $n$ is even in \cite[\S3.3]{Kottwitz}.  
On the other hand,
for an unordered bipartition $\{\alpha,\beta\}$ of $n$, Kottwitz \cite[\S3.3]{Kottwitz}  proves that the multiplicity of $\chi^{\{\alpha,\beta\}}$ in $\chi_{W,\omega_{k,\ell,m}}$ (assuming $\ell$ is even and $k+\ell \neq 0$)
is zero unless 
$\chi^{\{\alpha,\beta\}}$  
is special and  $\{ |\alpha|, |\beta|\} =\{ k+\ell+m,m\}$, in which case the  multiplicity  
is 
$\binom{e}{\ell+m-f_0}$, where 
\begin{enumerate}
\item[$\bullet$] $e$ is the number of integers in the sequence $\mu_1,\mu_2,\dots,\mu_n$ which do not occur in the sequence $\lambda_1,\lambda_2,\dots,\lambda_{n}$, where we define $\lambda_i = \alpha_{n-(i-1)} + (i-1)$ and $\mu_i = \beta_{n-(i-1)} + (i-1)$ as in (\ref{d-special}).

\item[$\bullet$] $f_0$ is $\sum_{i \in \PP} \max\{\beta_{i+1},\alpha_i\}$ if $\alpha \subset \beta$ or $ \sum_{i \in \PP} \max\{\alpha_{i+1},\beta_i\}$ if $\beta \subset \alpha$.
\end{enumerate}
In light of Proposition \ref{special-prop}, it suffices to show that if $\chi^{\{\alpha,\beta\}}$ is special then $\binom{e}{\ell+m-f_0}$ is equal to the coefficient $\binom{e(\alpha,\beta)}{f(\alpha,\beta)-\ell}$ defined in the theorem. 

We may assume  $\alpha\subsetneq \beta$ and $\beta\setminus \alpha$ contains no $2\times 2$ squares.
Then $\lambda_1 \leq \mu_1 \leq \lambda_2 \leq \mu_2 \leq \dots \leq \lambda_n \leq \mu_n$, and $e$ counts the number of $i \in [n]$ with $\lambda_{i} < \mu_i < \lambda_{i+1}$, where by convention $\lambda_{n+1} = \infty$. Equivalently, $e$ is  equal to the number of $i \in \PP$ with $\alpha_i < \beta_i \leq \alpha_{i-1}$, where we define $\alpha_0 = \infty$.
Consulting the picture in the remarks following Proposition \ref{special-prop}, one finds that the upper right most cells of the connected components of $\beta\setminus \alpha$
are precisely those  of the form $(i,\beta_i)$ for  $i \in \PP$ with $\alpha_i <\beta_i \leq \alpha_{i-1}$. Hence $e = e(\alpha,\beta)$.

Continuing, it follows from our assumptions that $f_0= |\alpha|  + g_0= m+g_0$, where $g_0$ is the number of $ i \in \PP$ for which $\alpha_i +1 = \beta_{i+1}$.  The number of connected components of $\beta\setminus \alpha$ is equal to $f(\alpha,\beta)$ (the number of $i\in \PP$ with $\alpha_i < \beta_i$) minus $g_0$, since $g_0$ counts the nonempty rows in $\beta\setminus \alpha$ which are connected to the rows above.  Thus $\binom{e}{\ell+m-f_0} = \binom{e}{\ell - g_0} = \binom{e}{e+g_0 - \ell} = \binom{e(\alpha,\beta)}{f(\alpha,\beta) - \ell}$, as desired.
\end{proof}

\subsection{The decomposition of $\chi_W$ in type $H_3$}\label{h3-sect}

Let $(W,S)$ be the Coxeter system of type $H_3$, so that $W$ is isomorphic to the direct product of $S_2$ and the alternating subgroup of $S_5$, and  $S = \{a,b,c\}$ consists of three elements, which we label according to the Dynkin diagram \[
\xy\xymatrix{
   a  \ar @{-}^{5} [r] & b \ar @{-}^{} [r] & c   
}\endxy
\] (so that $a,b,c$ satisfy the relations $(ab)^5 =  (ac)^2 = (bc)^3 = a^2=b^2=c^2 = 1$).
  There are then four conjugacy classes of involutions in $W$, represented by the elements $1$, $a$, $ac$, and $(abc)^5$, and 
$W$ has 10 irreducible characters. We refer to these, mirroring  Carter's notation in \cite{C}, by writing $\phi_{d,e}$ for the irreducible character of $W$ satisfying
\[ \phi_{d,e}(1) =d \qquad\text{and}\qquad \FakeDeg(\phi_{d,e}) = (\text{nonzero constant})\cdot x^e + \text{higher order terms}.
\] This convention manages to identify each element of $\Irr(W)$ uniquely.

Decomposing $\varrho_W$ and its subrepresentations  is a routine calculation which we have carried out in the computer algebra system {\sc{Magma}} \cite{magma}.  From this computation, we obtain the following:

\begin{proposition}\label{h3-tbl}
If $(W,S)$ is of type $H_3$ then the characters $\chi_{W,\sigma}$ decompose as follows:
\begin{enumerate}
\item[(1)] $\chi_{W,1} = \phi_{1,0}$.

\item[(2)] $\chi_{W,(abc)^5} = \phi_{1,15}$.

\item[(3)] $\chi_{W,a} =  \phi_{3,1} + \phi_{3,3}  +\phi_{4,3} + \phi_{5,5}$.

\item[(4)] $\chi_{W,ac} = \phi_{3,6} + \phi_{3,8} + \phi_{4,4} + \phi_{5,2}$.

\end{enumerate}
Consequently, $\chi_W = \sum_{\psi \in \Irr(W)} \psi$ and $\varrho_W$ is a Gelfand model.
\end{proposition}



%

\subsection{The decomposition of $\chi_W$ in type $H_4$}
\label{h4-sect}

Let $(W,S)$ now be the Coxeter system of type $H_4$, so that $W$ is  a finite group of order 14400 generated by four elements $S = \{a,b,c,d\}$, which we label according to the Dynkin diagram
 \[
\xy\xymatrix{
   a  \ar @{-}^{5} [r] & b \ar @{-}^{} [r] & c   \ar @{-}^{} [r] & d   
}\endxy.
\] 
%
%
There are five conjugacy classes of involutions in $W$, represented by 
$1$, $a$, $ac$, $(abc)^5$, and $(abcd)^{15}$, and $\Irr(W)$ has 34 elements.
We refer to the irreducible characters of $W$ exactly as in type $H_3$, with one exception: in type $H_4$, the notation $\phi_{d,e}$ fails to distinguish the two
irreducible characters of $W$ of degree 30,  and we consequently denote these  characters by  $\phi_{30,10,12}$ and $\phi_{30,10,14}$: here, $\phi_{30,10,f}$ indicates  the irreducible character with 
\[ \phi_{30,10,f}(1) = 30\qquad\text{and}\qquad \FakeDeg(\phi_{30,10,f}) = x^{10} + x^f + \text{higher order terms}.\]
Again, decomposing $\chi_W$ and $\chi_{W,\sigma}$ is a routine  computation, which we have carried out in {\sc{Magma}}. 

\begin{proposition}\label{h4-tbl}
If $(W,S)$ is of type $H_4$ then the characters $\chi_{W,\sigma}$ decompose as follows:
\begin{enumerate}
\item[(1)] $\chi_{W,1} = \phi_{1,0}$.

\item[(2)] $\chi_{W,(abcd)^{15}} = \phi_{1,60}$.

\item[(3)] $\chi_{W,a} =  \phi_{4,1} + \phi_{4,7}  +\phi_{16,3} + \phi_{36,5}$.

\item[(4)] $\chi_{W,(abc)^5} = \phi_{4,31} + \phi_{4,37} + \phi_{16,21} + \phi_{36,15}$.

\item[(5)] $\ba \chi_{W,ac} =\ & \phi_{9,2} + \phi_{9,6} + \phi_{9,22} + \phi_{9,26} + \phi_{16,6} + \phi_{16,18} + \phi_{25,4} + \phi_{25,16} \\ & +  2\phi_{24,6}  +2 \phi_{24,12} +  2\phi_{18,10} +  2\phi_{30,10,12} +2 \phi_{30,10,14} +  2\phi_{40,8}.\ea$
\end{enumerate}
\end{proposition}

\subsection{The decomposition of $\chi_W$ in type $I_2(m)$}\label{i2-sect}

Fix $m\geq 3$ and 
 let $(W,S)$ be the Coxeter system of type $I_2(m)$, so that
$W$ is  the dihedral group of order $2m$, generated by the two elements $S=\{r,s\}$ subject to the relations $r^2=s^2 =(rs)^m =1$.
The involutions in $W$ are the elements 1 and $(rs)^jr$ for $0\leq j \leq m-1$, along with $(rs)^{m/2}$ if $m$ is even.
 One checks that
\be\label{des-i2}
\Des(1) = \varnothing,
\qquad
\Des\((rs)^{m/2}\) = \{r,s\},
\qquad
\Des\((rs)^jr\) = \begin{cases} \{r\},&\text{if } 2j+1<m, \\
 \{r,s\},&\text{if }2j+1=m, \\
 \{s\},&\text{if }2j+1>m.\end{cases}
\ee
Write $w_0$ for the longest element of $W$, given by   $(rs)^{\frac{m-1}{2}}r$ if $m$ is odd or $(rs)^{m/2}$ if $m$ is even.
The group $W$  has either two or four conjugacy classes of involutions, represented by $1$ and $r$ if $m$ is odd and by $1$, $w_0$, $r$, and $s$ if $m$ is even.

The irreducible characters of $W$ are given as follows. There are two linear characters when $m$ is odd, given by $\phi_{1,0} = \One$ and $\phi_{1,m} = \sgn$, and four linear characters when $m$ is even, given by 
\[\phi_{1,0}=\One, \qquad \phi_{1,m} = \sgn, \qquad \phi'_{1,m/2}: r^js^k \mapsto (-1)^k,\qquad\text{and}\qquad \phi_{1,m/2}'' : r^js^k \mapsto (-1)^j.\]
There are in addition $\lfloor\frac{m-1}{2}\rfloor$ distinct irreducible characters of degree two given by the  functions 
\[ \barr{rlcl} \phi_{2,k} : & (rs)^jr &\mapsto& 0, \\ &(rs)^j &\mapsto& 2\cos(2\pi jk/m),
\earr
\qquad\text{for integers $k$ with } 0<k<\frac{m}{2}.
\]
These constructions exhaust all elements of $\Irr(W)$. 
We have labeled the characters of $W$ following our convention in type $H_3$ and $H_4$: the first index of $\phi_{d,e}$ indicates the character's degree while the second index is the largest power of $x$ dividing the character's fake degree. 
In the notation of Section \ref{i2-unipotent}, we have 
\[ \phi_{1,0} =\One,\qquad \phi_{1,m} =\sgn,
 \qquad
\phi_{1,m/2}' = \Phi_{(0,\frac{m}{2})}',
\qquad
\phi_{1,m/2}'' = \Phi_{(0,\frac{m}{2})}'',
\qquad
\phi_{2,k} = \Phi_{(0,k)}.
\]


%
%
%

Table \ref{i2-tbl} describes the irreducible decomposition of the characters $\chi_{W,\sigma}$ and $\chi_{W}$ in type $I_2(m)$. 
 The rows in this table  correspond to individual irreducible characters of $W$, 
 while the columns list the multiplicity of each row in the characters $\chi_{W,\sigma}$.
 
Three distinct patterns arise according to the residue class of $m$ modulo 4. The proof of the given decompositions
is a simple exercise using Frobenius reciprocity and the fact that  $\chi_{W,\sigma}$ is induced from a linear character $\lambda$ of  the centralizer $C_W(\sigma)$ of $\sigma$ in $W$.
The character $\lambda$ is determined by formula $\varrho_W(g) a_\sigma = \lambda(g) a_\sigma$ for $g \in C_W(\sigma)$, which may be explicitly evaluated using (\ref{des-i2}).  In addition, one checks  that if $m$ is odd and $\sigma$ is one of the representative involutions $1$ or $r$ then $C_W(\sigma)$ is $W$ or $\{1,r\} \cong S_2$, and that 
if 
$m$ is even and $\sigma$ is $1$, $w_0$, $r$, or $s$
then $C_W(\sigma)$ is  $W$, $W$, $\{ 1,w_0, r,w_0 r\} \cong S_2\times S_2$, or $\{1,w_0,s,w_0s\} \cong S_2 \times S_2$.  
Evaluating the inner product of  $\lambda$  with the  elements of  $ \Irr(W)$ restricted to these subgroups  provides the desired multiplicities of $\chi_{W,\sigma}$.

Summarizing Table \ref{i2-tbl}, we have the following proposition. 

\begin{proposition}\label{i2-prop}
Suppose $(W,S)$ is of type $I_2(m)$ with $m\geq 3$. 

\begin{enumerate}
\item[(1)] If $m$ is odd then $\chi_{W} =  \phi_{1,0} + \phi_{1,m} + \sum_{k=1}^{\frac{m-1}{2}} \phi_{2,k}=
 \sum_{\psi \in \Irr(W)} \psi$.

\item[(2)] If $m \equiv 2\modu 4)$ then 
$\chi_{W} =   \phi_{1,0} + \phi_{1,m}+ \phi_{1,m/2} + \phi_{1,m/2}' + \sum_{k=1}^{\frac{m-2}{4}} 2\phi_{2,2k-1}.$

\item[(3)] If $m \equiv 0\modu 4)$ then $\chi_{W} = \phi_{1,0} + \phi_{1,m} + \sum_{k=1}^{\frac{m}{4}} 2\phi_{2,2k-1}$.

\end{enumerate}
\end{proposition}

%

\begin{table}[b]
\[
\ba
&\ba 
&m\text{ odd}
\\
& \barr{l | ll | l}\hline
\text{Character}\quad &  
\chi_{W,1} & \chi_{W,r} & \chi_W
\\
\hline
\phi_{1,0} & 1 & 0 & 1
\\
\phi_{1,m} & 0 & 1 & 1
\\
\phi_{2,k}\ (0<k<\tfrac{m}{2}) & 0 & 1 & 1
\\
\hline\earr
\ea
\\
\\
&\ba
&m\equiv 2 \modu 4)
\\
& \barr{l | llll | l}\hline
\text{Character}\quad &  
\chi_{W,1} & \chi_{W,w_0} & \chi_{W,r} & \chi_{W,s} & \chi_W
\\
\hline
\phi_{1,0} & 1 & 0 & 0 & 0 & 1
\\ 
\phi_{1,m} & 0 & 1 &  0 & 0 & 1
\\
\phi_{2,k}\ (k\text{ odd}) & 0 & 0 & 1 & 1 & 2 
\\
\phi_{2,k}\ (k\text{ even}) & 0 & 0 & 0 & 0 & 0 
\\
\phi'_{1,m/2} & 0 & 0 & 0 & 1 & 1
\\
\phi''_{1,m/2} & 0 & 0 & 1 & 0 & 1
\\
\hline\earr
\ea
\\
\\
&
\ba &m \equiv 0 \modu 4)
\\
& \barr{l | llll | l}\hline
\text{Character}\quad &  
\chi_{W,1} & \chi_{W,w_0} & \chi_{W,r} & \chi_{W,s} & \chi_W
\\
\hline
\phi_{1,0} & 1 & 0 & 0 & 0 & 1
\\
\phi_{1,m} & 0 & 1 &  0 & 0 & 1
\\
\phi_{2,k}\ (k\text{ odd}) & 0 & 0 & 1 & 1 & 2 
\\
\phi_{2,k}\ (k\text{ even}) & 0 & 0 & 0 & 0 & 0 
\\
\phi'_{1,m/2} & 0 & 0 & 0 & 0 & 0
\\
\phi''_{1,m/2} & 0 & 0 & 0 & 0 & 0
\\
\hline\earr
\ea
\ea\] 
\caption{Irreducible multiplicities of $\chi_{W,\sigma}$ and $\chi_{W}$ in type $I_2(m)$
}
\label{i2-tbl}
\end{table}

\section{Fourier transform matrices and the proof of Theorem \ref{main-thm}}
\label{FT-sect}

\def\T{\textbf{F}}

Here we describe the Fourier transform matrices associated to $\Uch(W)$ for each finite, irreducible 
Coxeter system $(W,S)$, and derive from this setup the proof of Theorem \ref{main-thm}.
In  the crystallographic case, the relevant definitions are well-established and due originally to Lusztig \cite{L}. 
We must take more care to describe the associated matrices for the non-crystallographic types, as the literature \cite{Zyk,L_exotic,Malle74,Malle_Imprimitiver} presenting this heuristic theory is not nearly as cohesive or extensive.

We  first describe how to attach to each family  in $ \Uch(W)$ (see Section \ref{family-sect}) a Fourier transform matrix $\FT$. 
The Fourier transform matrix of $\Uch(W)$ is  subsequently constructed as the direct sum of such matrices $\FT$ over all families.
In Section \ref{ft-last-sect} we  describe some notable properties of this Fourier transform and discuss in what sense  these properties indicate the choices of matrices $\FT$ to be ``canonical.''

\subsection{Fourier transform matrices in crystallographic types}
\label{cry-ft-sect}

As described in Section \ref{family-sect}, when $(W,S)$ is crystallographic, every family 
in $\Uch(W)$ is parametrized by a set $\sM(\Gamma)$ for some finite group $\Gamma$. (Recall from (\ref{MGamma})  the definition of this set.) Lusztig \cite{L} defines the Fourier transform associated to a 
  family indexed by $\sM(\Gamma)$ as the matrix
\be\label{ftgamma} \FT_\Gamma \omdef= \( \{ m,m'\}\)_{m,m' \in \sM(\Gamma)}\ee whose entries are the numbers
\[ \{ (x,\sigma),(y,\tau)\} \omdef = \frac{1}{|C_\Gamma(x)|} \frac{1}{|C_\Gamma(y)|} \sum_{\substack{g \in \Gamma \\
x\cdot gyg^{-1} = gyg^{-1}\cdot x}}
\sigma\(gyg^{-1}\) \tau\(g^{-1}x^{-1}g\)
\] for $(x,\sigma),(y,\tau) \in \sM(\Gamma)$.
Carter helpfully provides an explicit description of the sets $\sM(\Gamma)$ and the accompanying matrices $\FT_\Gamma$ in the  cases when $\Gamma = S_n$ and $n\leq 4$ (as well a partial matrix in the case $\Gamma = S_5$)  \cite[\S13.6]{C}; see also the overview in \cite[\S12.3]{C}.
We review the frequently occuring case $\Gamma = S_2$ in the following example.

\begin{example}\label{FD-ex}
 If  $s$ denotes the nontrivial element of $S_2$, then $\sM(S_2) = \{ (1,\One), (1,\sgn), (s,\One), (s,\sgn) \}$, 
and  with respect to the order in which we just listed $\sM(S_2)$, the corresponding matrix  is
\[ 
\FT_{S_2} = \frac{1}{2}\(\barr{rrrr} 1 & 1 & 1 &1 \\ 1 & 1 & -1 & -1 \\ 1 & -1 & 1 & -1 \\ 1 & -1 & -1 & 1 \earr\).
\]
\end{example}

\def\D{\textbf{D}}
\subsection{Fourier transform matrices in type $I_2(m)$}\label{i2ft}

The description of the Fourier transform in this case comes from Lusztig's paper \cite{L_exotic}.
Assume $(W,S)$ is of type $I_2(m)$ for an integer $m\geq 3$ and recall the explicit construction of $\Uch(W)$ given in Section \ref{i2-unipotent}.  As noted there, $\Uch(W)$ has only three families, two of which have size one: $\{ \One\}$ and $\{ \sgn\}$.  The Fourier transform  of both 1-element families is the $1\times 1$ identity matrix.

Let $\cF = \Uch(W)\setminus \{ \One,\sgn\}$ denote the remaining family, and define $X = X' \cup X''$, where $X'$ and $X''$ are the disjoint sets given by
\[\ba X' &= \Bigl\{\text{Pairs of integers $(i,j)$ with $i+j<m$ and either $0<i<j<m$ or $0=i<j < \tfrac{m}{2}$}\Bigr\},
\\
 X'' &= \begin{cases} \{(0,\tfrac{m}{2})',(0,\tfrac{m}{2})''\},&\text{if $m$ is even}, \\ \varnothing,&\text{otherwise}.\end{cases}
 \ea
 \]
 As is clear from our notation in Section \ref{i2-unipotent}, $X$ naturally parametrizes $\cF$.  Write $\xi= \exp\(\frac{2\pi \sqrt{-1}}{m}\)$
for the standard $m^{\mathrm{th}}$ root unity.
The Fourier transform of $\cF$, as defined by Lusztig \cite{L_exotic},
is then the matrix
 \[ \D_m \omdef= \( \{ x,x'\}\)_{x,x' \in X}\]  whose entries are the numbers $\{x,x'\}$ 
 given by 
\[\{(i,j),(k,l)\} = \tfrac{1}{m} \( \xi^{jk-il} + \xi^{-jk+il} - \xi^{-ik+jl} - \xi^{ik-jl}\),\qquad\text{for $(i,j),(k,l) \in X'$},\]
and if $m$ is even and $(i,j) \in X'$, by
\[\{(i,j),(0,\tfrac{m}{2})'\}  = \{(i,j),(0,\tfrac{m}{2})''\}=\{(0,\tfrac{m}{2})',(i,j)\}  = \{(0,\tfrac{m}{2})'',(i,j)\} = \tfrac{(-1)^i - (-1)^j}{m},\]
\[
\{(0,\tfrac{m}{2})',(0,\tfrac{m}{2})'\} = \{(0,\tfrac{m}{2})'',(0,\tfrac{m}{2})''\} = \tfrac{1-(-1)^{m/2}}{2m} +  \tfrac{1}{2},\]
\[
\{(0,\tfrac{m}{2})',(0,\tfrac{m}{2})''\} = \{(0,\tfrac{m}{2})'',(0,\tfrac{m}{2})'\} = \tfrac{1-(-1)^{m/2}}{2m} -  \tfrac{1}{2}  .
\]
We have labeled this  matrix $\D_m$ as a mnemonic for  ``dihedral Fourier transform.''

\begin{example}\label{d-ex}
It is helpful to review some examples of this construction.
\begin{enumerate}
\item[(i)] If $m=3$ then $X = \{ (0,1)\}$ and $\D_3$ is the $1\times 1$ identity matrix.

\item[(ii)] If $m=4$ then $X = \{ (0,1),(0,2)',(0,2)'',(1,2)\} $ and 
with respect to an appropriate ordering of indices we have $\D_4 = \FT_{S_2}$.
In a similar way, one finds that $\D_6 = \FT_{S_3}$. 
These equalities are consistent with the fact that the Coxeter systems of types $I_2(4)$ and $I_2(6)$ are isomorphic to those of types  $BC_2$ and $G_2$, whose nontrivial families  of unipotent characters are parametrized by $\sM(S_2)$ and $\sM(S_3)$.

\item[(iii)] If $m=5$ then $X = \{ (0,1), (0,2), (1,2), (1,3)\}$ and   with respect to the order in which we just listed $X$, one computes
\[ \D_5 = \frac{1}{\sqrt{5}} \(\barr{rrrr} \wt \lambda &\lambda & 1 & 1 \\ \lambda & \wt \lambda & -1 & -1 \\ 1 & -1 & \lambda & -\wt\lambda \\ 1 & -1 & -\wt\lambda & \lambda \earr\), 
\qquad\text{where }
 \lambda = \frac{\sqrt{5}+1}{2}
 \text{ and } \wt\lambda =\frac{\sqrt{5}-1}{2}.\]
 This is precisely the matrix listed in \cite[Eq.\ (7.3)]{Zyk} and in \cite[\S3.10]{L_exotic}.


\end{enumerate}
 \end{example}

\subsection{Fourier transform matrices in types $H_3$ and $H_4$}\label{h-ft-sect}

The definition of the Fourier transform in the remaining non-crystallographic cases comes from the papers \cite{Zyk,Malle74,Malle_Imprimitiver}.
In type $H_3$, $\Uch(W)$ has four 1-element families given by 
$\{ \phi_{1,0}\},$  $\{ \phi_{1,15}\},$ $\{ \phi_{5,2}\},$ and $\{ \phi_{5,5}\}.$
In type $H_4$, $\Uch(W)$ has six 1-element families given by 
$ \{ \phi_{1,0}\},$ $ \{ \phi_{1,60}\},$ $ \{ \phi_{25,4}\},$ $ \{ \phi_{25,16}\},$ $
\{ \phi_{36,5}\},$ and $ \{ \phi_{36,15}\}.$
These  families are all subsets of $\Irr(W)$ and we  have  labeled their (necessarily special) elements according to our conventions in Sections \ref{h3-sect} and \ref{h4-sect}.  To all such 1-element families, the associated Fourier transform matrix is the $1\times1 $ identity matrix.


In type $H_3$ (respectively, $H_4$), $\Uch(W)$ has three (respectively, six) 4-element families.  One such family is 
 exceptional in type $H_3$ (see Definition \ref{ex-def}),
 while two are exceptional in type $H_4$.  The following observations concerning these families are derived from the parametrizations of $\Uch(W)$ provided by the {\tt{UnipotentCharacters}} command in CHEVIE \cite{CHEVIE}.   
First, in both types
the exceptional families always consist of four elements 
\[\Phi_{(1,\One)},\quad  
\Phi_{(1,\sgn)},\quad
\Phi_{(s,\One)},\quad
\Phi_{(s,\sgn)}\]
which can be indexed by the set $\sM(S_2)$, such that 
\begin{enumerate}
\item[$\bullet$] $\Phi_{(1,\One)}, \Phi_{(1,\sgn)} \in \Irr(W)$, with $\Phi_{(1,\One)}$ special and $\Deg\(\Phi_{(1,\One)}\) = \Deg\(\Phi_{(1,\sgn)}\) $.
\item[$\bullet$] $\Phi_{(s,\One)}, \Phi_{(s,\sgn)} \notin \Irr(W)$, with 
$\Deg\(\Phi_{(s,\One)}\) = \Deg\(\Phi_{(s,\sgn)}\)$ and 
$\begin{cases} \Eig\(\Phi_{(s,\One)}\) = i, \\ \Eig\(\Phi_{(s,\sgn)}\) = -i. \end{cases}$

\end{enumerate}
\begin{remark}
The computer algebra system CHEVIE stores a large amount of data associated to $\Uch(W)$, which can be accessed by combining the commands {\tt{Display}} and {\tt{UnipotentCharacters}}. The parametrization by $\sM(S_2)$ just given, however, is not  
included in CHEVIE, though the listed properties uniquely determine which index $(x,\sigma) \in \sM(S_2)$ goes to which $\Phi \in \cF$ for any exceptional 4-element family $\cF$.
\end{remark}
The characters $\Phi_{(1,\One)}, \Phi_{(1,\sgn)} \in \Irr(W)$ may be respectively either $\phi_{4,3}$, $\phi_{4,4}$ in type $H_3$ or
$\phi_{16,3}$, $\phi_{16,6}$ or $\phi_{16,18}$, $\phi_{16,21}$ in type $H_4$. There is no established notation for the formal elements $\Phi_{(s,\One)}$, $\Phi_{(s,\sgn)}$  in each family, however.
 
The non-exceptional 4-element families in types $H_3$ and $H_4$, on the other hand, always consist of four elements 
\[\Phi_{(0,1)}, \quad\Phi_{(0,2)}, \quad\Phi_{(1,2)}, \quad\Phi_{(1,3)}\] which can be indexed by the set $X$ in Section \ref{i2ft} with $m=5$, such that if $\xi = \exp \(\frac{2\pi \sqrt{-1}}{5}\)$ is a fifth root of unity, then
\begin{enumerate}
\item[$\bullet$] $\Phi_{(0,1)}, \Phi_{(0,2)} \in \Irr(W)$, with $\Phi_{(0,1)}$ special
and $\Deg\(\Phi_{(0,1)}\) \neq \Deg\(\Phi_{(0,2)}\) $.
\item[$\bullet$] $\Phi_{(1,2)}, \Phi_{(1,3)} \notin \Irr(W)$, with 
$\Deg\(\Phi_{(1,2)}\) = \Deg\(\Phi_{(1,3)}\)$ and 
$\begin{cases} \Eig\(\Phi_{(1,2)}\) = \xi^3, \\ \Eig\(\Phi_{(1,3)}\) = \xi^2. \end{cases}$
\end{enumerate}
  This parametrization by $X$, though uniquely determined for each non-exceptional family, is again not actually 
listed in CHEVIE.
    The characters $\Phi_{(0,1)}, \Phi_{(0,2)} \in \Irr(W)$ may be either $\phi_{3,1}$, $\phi_{3,3}$ or $\phi_{3,6}$, $\phi_{3,8}$ in type $H_3$ or any of the pairs $\phi_{4,1}$, $\phi_{4,7}$ or $\phi_{4,31}$, $\phi_{4,37}$ or  $\phi_{9,2}$, $\phi_{9,6}$ or $\phi_{9,22}$, $\phi_{9,26}$  in type $H_4$. 
  There is again no established notation for the formal elements $\Phi_{(1,2)}$, $\Phi_{(1,3)}$ 
 in each family. 
 
 The Fourier transforms of these families are now defined thus: if  $\cF\subset \Uch(W)$ is a 4-element family in type $H_3$ or $H_4$,
  then its Fourier transform matrix is 
  \be\FT = \begin{cases} \FT_{S_2},&\text{if $\cF$ is exceptional (see Example \ref{FD-ex})}, \\
  \D_5,&\text{if $\cF$ is not exceptional (see Example \ref{d-ex}(iii)).}\end{cases}\ee
Note that these assignments  make  sense because we have indicated how each 4-element family is indexed by the same set as the corresponding matrix.

\begin{remark}
In our definition of the Fourier transform matrix for the non-exceptional 4-element families, we follow the convention of \cite[\S7]{Zyk}. The Fourier transform of the exceptional 4-element families in types $H_3$ and $H_4$ seems less well-established in the literature. The matrix  assigned  to these families here  is chosen to be identical to the Fourier transform matrix of the other exceptional families in types $E_7$ and $E_8$. We will say more about the ``correctness'' of this choice in the remarks following Theorem \ref{p4} below.
\end{remark}

These conventions attach a Fourier transform matrix to all but  one remaining family in type $H_4$.
In this type, $\Uch(W)$ has a single  family $\cF$ of size 74; the  intersection of this family with $\Irr(W)$ has size 16 and its unique special element is the character $\phi_{24,6} \in \Irr(W)$.
The Fourier transform  of this family 
is constructed by 
Malle as the matrix $S$ in his paper  \cite{Malle74}.
To do calculations with this matrix, one  needs to be able to access it  in some computer format, and the  algebra package CHEVIE fortunately provides this capability. In detail,
 one can obtain the $74\times 74$ Fourier transform matrix  of  $\cF$  by the following sequence of CHEVIE commands in GAP:
\[ \ba &{\tt{W := CoxeterGroup(``H",4);}}
\\ &{\tt{Uch := UnipotentCharacters(W);}}
\\ &{\tt{F := Uch.families[13];}}
\\ &{\tt{M := F.fourierMat * MatPerm(F.perm,74);}}
\ea\]
The odd-looking multiplication by ${\tt{MatPerm(F.perm,74)}}$ in the last line 
has to do with the indexing conventions of the $\tt{fourierMat}$ field in CHEVIE.
The code given here produces a matrix $\tt{M}$ whose rows and columns have the same indices as $\tt{Uch}$, and which is identical to the one in \cite{Malle74}.

\subsection{Observations and consequences}\label{ft-last-sect}

From 
the preceding three subsections, we know of a Fourier transform matrix attached to each family $\cF$ in $\Uch(W)$ for each finite, irreducible Coxeter system $(W,S)$. The \emph{Fourier transform matrix} of $\Uch(W)$, we reiterate, is  the direct sum of these matrices over all families $\cF$. 
By construction, this $\Uch(W)$-indexed matrix satisfies property (P2) in the introduction; i.e., it is block diagonal with respect to the decomposition of $\Uch(W)$ into families.
The following theorem
explains precisely how 
 this matrix
also satisfies property (P1) in the introduction.
This statement should be attributed to Lusztig and Malle via a combination of results  appearing in the papers \cite{L,L_exotic,Malle74,Malle_Imprimitiver};
see the proof of \cite[Theorem 6.9]{GeckMalle} for a  detailed bibliography.

\begin{theorem}\label{p1} Let $(W,S)$ be a finite, irreducible Coxeter system with associated Fourier transform matrix $\FT$, and write $j$ for the (permutation matrix of the) involution of $\Uch(W)$ defined in Proposition \ref{j-prop}. Then the composition $\FT\circ j$ transforms for the vector of fake degrees of $\Uch(W)$ to the vector of (generic) degrees.
\end{theorem}

\begin{remark}
This result is mentioned in several places, but  sometimes  imprecisely. 
For example, Carter asserts, without any mention of $j$, that the Fourier transform matrix of $\Uch(W)$ transforms the vector of fake degees to the vector of actual degrees whenever $(W,S)$ is a Weyl group  \cite[\S13.6]{C}. This statement, at least as we interpret it, is not strictly true in types $E_7$ and $E_8$. In particular, one can check that the Fourier transform matrix $\FT$ that Carter attaches to the three exceptional families in these types (see Section \ref{family-sect}) does not literally transform the vector of fake degrees
to the corresponding vector of (generic) degrees as listed in \cite[\S13.8]{C}.  But $\FT$ does transform a nontrivial permutation of the fake degrees to the (generic) degrees.
\end{remark}


\def\F{\textbf{t}}

The preceding theorem gives one reason to consider the particular 
 Fourier matrices assigned to $\Uch(W)$  
as somehow ``canonical,'' and our next result provides another.
To state this, we  first must recall  Lusztig's definition of a \emph{fusion datum} \cite{L_exotic}.

\begin{definition}\label{fd-def}
Let $X$ be a finite set with a distinguished 
element $x_0$, and suppose 
\begin{enumerate}
\item[$\bullet$] $\Delta$ is  the matrix of an involutory permutation of $X$ with $x_0$ as a fixed point;
\item[$\bullet$] $\FT$ is  a real symmetric matrix indexed by $X$;
\item[$\bullet$] $\T$ is a diagonal  matrix indexed by $X$ whose diagonal entries are complex roots of unity.
\end{enumerate}
The  tuple $\(X,x_0,\Delta,\FT,\T\)$ is a \emph{fusion datum} if
the following axioms hold:
\begin{enumerate}
\item[](Commutability). $\FT = \Delta \FT\Delta$  and $\T^{-1} = \Delta \T \Delta$.
\item[](Positivity). $\FT_{x,x_0} > 0$ for all $x \in X$ and $\T_{x_0,x_0} = 1$.
\item[](Modularity). $\FT^2 = (\T\Delta\FT)^3 = 1$.
\item[](Integrality). $\sum_{w \in X} \frac{\FT_{x,w} \FT_{y,w} \FT_{z,w}}{\FT_{x_0,w}} \in \NN$ for all $x,y,z \in X$.

\end{enumerate}
\end{definition}

\begin{remarks} A few comments are helpful in unpacking this not altogether transparent construction.
\begin{enumerate}
\item[(a)]
 In \cite{L_exotic}, Lusztig actually presents a more general definition of a fusion datum which involves two involutions $\sharp$, $\flat$ of $X$ in place of $\Delta$. This  definition reduces to ours  when $\sharp = \flat = \Delta$.
 \item[(b)] The modularity axiom is so-named as it  requires that the matrices $\Delta$, $\T$,  $\FT$ determine  a unitary representation of the modular group $\PSL_2(\ZZ)$.  

\item[(c)] The integrality axiom allows one to define an algebra  structure with nonnegative integer structure coefficients on the complex vector space generated by $X$. More specifically, the axiom asserts that one can give this vector space the structure of a \emph{based ring} in the sense of \cite{L3}; see the discussion in \cite[\S7]{GeckMalle}.
\end{enumerate}
  \end{remarks}

Before proceeding we note
the following short lemma, which identifies a common type of fusion datum based on the set $\sM(\Gamma)$ attached to a finite group $\Gamma$.

 \begin{lemma}\label{fd-ex-lem} Let $\Gamma$ be  a finite group equal to $(\ZZ/2\ZZ)^n$ or $S_n$ for some  $n\geq 0$, and suppose
\begin{enumerate}
\item[$\bullet$] $X=\sM(\Gamma)$ as defined by (\ref{MGamma})
and  $x_0 = (1,\One) \in X$;

\item[$\bullet$] $\Delta : X\to X$ is the involution defined by $\Delta : (x,\sigma) \mapsto (x,\overline{\sigma})$;

\item[$\bullet$]  $\FT = \FT_\Gamma$ as in (\ref{ftgamma}); 

\item[$\bullet$]  $\T = \diag (t_x)_{x \in X}$ where $t_{(x,\sigma)} = \frac{\sigma(x)}{\sigma(1)}$ for $(x,\sigma) \in \sM(\Gamma)$.
\end{enumerate}
Then $(X,x_0,\Delta,\FT,\T)$ is a fusion datum in the sense of Definition \ref{fd-def}.
\end{lemma}

\begin{proof}
This is equivalent to \cite[Proposition 1.6]{L_exotic}, provided we show 
 that $ (x,\overline\sigma)\sim (x^{-1},\sigma)$ for each  $(x,\sigma) \in \sM(\Gamma)$ 
when $\Gamma$ is either $(\ZZ/2\ZZ)^n$ or $S_n$. This is immediate in the first case, so assume $\Gamma = S_n$ is a symmetric group.

Let $\lambda$ be a partition of $n$ and let $x \in S_n$ be a permutation with cycle type $\lambda$.
In the proof of Proposition \ref{torics-prop}, we noted that $C_{S_n}(x)$ is isomorphic to  the direct product of wreath products $\prod_{r\geq 1} G(r,m_r)$, where $m_r$ is the number of parts of $\lambda$ with size $r$.  
A less explicit version of this statement goes as follows: the cycles of $x$ generate an abelian group $A \cong \prod_{i\geq 1} \ZZ/\lambda_i\ZZ$, and $C_{S_n}(x)$ is isomorphic to a semidirect product of the form $\(S_{m_1} \times S_{m_2} \times \dots\) \ltimes A$.
It is easy to see that one can choose a permutation $g \in S_n$ which commutes with the left factor of this semidirect product and which has $gxg^{-1} = x^{-1}$ and in fact $g^{-1}yg = y^{-1}$ for all $ y \in A$.
Noting these properties, it follows from  the standard construction of the irreducible characters of a semidirect product with an abelian normal subgroup (see \cite[Exercise XVIII.7]{Lang}) that $\sigma^g = \overline \sigma$ for any $\sigma \in \Irr(C_{S_n}(x))$.   
We conclude that $(x,\overline\sigma) = (gxg^{-1},\overline\sigma^g) = (x^{-1},\sigma)$ for all $(x,\sigma) \in \sM(S_n)$, as required.
\end{proof}

With these preliminaries in tow, we may now state  
the following noteworthy result due to  Geck, Lusztig, and Malle \cite{GeckMalle,L_exotic, Malle74}, showing how each family in $\Uch(W)$, combined with its attached Fourier transform matrix, 
possesses naturally the structure of a fusion datum.
This observation elaborates property (P4) of the Fourier transform matrix noted in the introduction, and is essentially a special case of \cite[Theorem 6.9]{GeckMalle} (although as stated, that result drops the positivity axiom of a fusion datum and uses somewhat different terminology.)

\begin{theorem} \label{p4}
Let $(W,S)$ be a finite, irreducible Coxeter system.  Suppose
\begin{enumerate}
\item[$\bullet$] $\cF$ is a family in $\Uch(W)$ and $\Phi_0 \in \cF$ is its unique special element;
\item[$\bullet$] $\Delta$ is the restriction to $\cF$ of the involution of $\Uch(W)$ given in Proposition \ref{cmplx-prop};
\item[$\bullet$] $\FT$ is the Fourier transform matrix of $\cF$, as defined in the Sections \ref{cry-ft-sect}, \ref{i2ft},  \ref{h-ft-sect};
\item[$\bullet$] $\T = \diag\(\Eig(\Phi)\)_{\Phi \in \cF}$ is the diagonal matrix of Frobenius eigenvalues of $\Phi \in \cF$.
\end{enumerate}
Then $(\cF,\Phi_0, \Delta,\FT,\T)$ is a fusion datum.
\end{theorem}

\begin{remark}

For almost all families in $\Uch(W)$, our assignment of Fourier transform matrix follows exactly   the convention established in  the computer algebra system CHEVIE \cite{CHEVIE}.  For the six exceptional families in types $E_7$, $E_8$, $H_3$, and $H_4$, however, our assigned matrix  differs slightly from the one stored in CHEVIE$-$although, these matrices are the same after a permutation of rows and/or columns.  
In justification of our assignments, we can say the following: if one assumes $\Uch(W)$ given, then for each exceptional family, there is a \emph{unique} matrix $\FT$ satisfying both Theorem \ref{p1} and Theorem \ref{p4}. So at least in this sense our choice of $\FT$ is canonical.

\end{remark}

We include a brief proof of the  theorem for completeness.

\begin{proof}
In the case that $(W,S)$ is crystallographic and $\cF$ is a non-exceptional family of $\Uch(W)$, the theorem follows by combining Observation \ref{eig-obs} 
with Lemma \ref{fd-ex-lem}. If $\cF$ is one of the six exceptional families in types $E_7$, $E_8$, $H_3$, or $H_4$, then since $\cF$ is $\Delta$-invariant, the theorem 
is equivalent to the claim that $(X,x_0,\Delta,\FT,\T)$ is a fusion datum for
$X = \{1,2,3,4\}$ and $x_0 = 1$ and
\[ 
\Delta =\(\barr{rrrr} 1 & 0 & 0 &0 \\ 0 & 1 & 0 & 0 \\ 0 & 0 & 0 & 1 \\ 0 & 0 & 1 & 0 \earr\),
\qquad
\FT  = \frac{1}{2}\(\barr{rrrr} 1 & 1 & 1 &1 \\ 1 & 1 & -1 & -1 \\ 1 & -1 & 1 & -1 \\ 1 & -1 & -1 & 1 \earr\),
\qquad 
\T = \diag\(1,1,i,-i\).\] The proof of this is a straightforward computer calculation.
All families not covered by these cases occur when $(W,S)$ is non-crystallographic, and are either singletons; non-exceptional 4-element families in type $H_3$ or $H_4$; the nontrivial family in type $I_2(m)$; or the  74-element family in type $H_4$.
In the first case the theorem holds trivially; in the second and third case, the tuple $(\cF,\Phi_0, \Delta,\FT,\T)$ coincides with dihedral fusion datum given in \cite[\S3]{L_exotic}; and in the last case, $(\cF,\Phi_0, \Delta,\FT,\T)$ is by construction the fusion datum which  Malle describes in \cite{Malle74}.
\end{proof}

We conclude this section  
finally
by 
proving our main theorem 
from the introduction.

\begin{proof}[Proof of Theorem \ref{main-thm}]
If $(W,S)$ is classical, then the function $\fsd : \Uch(W)\to\RR$ which is identically 1 satisfies (1)-(3) in Theorem \ref{main-thm} and also (\ref{all-note}) by \cite[Theorem 1]{Kottwitz}.  The uniqueness of this function $\fsd$ follows by Theorem \ref{p4} as a consequence of the positivity axiom of a fusion datum, since any  function $\fsd' : \Uch(W)\to\RR$ satisfying (1) and (2) must have $\fsd'(\Phi) \leq \fsd(\Phi)$ for all $\Phi \in \Uch(W)$.
The final statement in the theorem holds because the Frobenius eigenvalues of $\Uch(W)$ all real if $(W,S)$ is classical by
 Observation \ref{eig-obs} (noting that the group $\Gamma$ in Observation \ref{eig-obs} is $(\ZZ/2\ZZ)^k$ in this case). 

Assume $(W,S)$ is of type $I_2(m)$, let $\FT$ denote the Fourier transform matrix of $\Uch(W)$, and let $\fsd :\Uch(W)\to \RR$ be the function with $\fsd (\Phi) = 0$ if $\Phi = \Phi_{(i,j)}$ for some $0<i<j<i+j < m$ such that $j\neq \frac{m}{2}$ and with $\fsd(\Phi) = 1$ for all other $\Phi \in \Uch(W)$.
From the discussion in Section \ref{i2-unipotent}, this function satisfies conditions (1) and (2) in the theorem. To show that it also satisfies (3) and (\ref{all-note}), it suffices by Proposition \ref{i2-prop} to check that if $\xi = \exp\( \frac{2\pi \sqrt{-1}}{m}\)$ and $0<i<j<i+j<m$, then when $m$ is odd, we have
\[ \FT \fsd\(\Phi_{(i,j)}\) \omdef= \sum_{0<k < \frac{m}{2}} \{ (i,j),(0,k) \} = \sum_{0<k < \frac{m}{2}} \tfrac{1}{m} \(\xi^{-ik} + \xi^{ik} -\xi^{jk} - \xi^{-jk}\) =\begin{cases} 1,&\text{if }i=0,\\0,&\text{otherwise},
\end{cases}\]
and when $m$ is even, we have
\[\ba \FT \fsd\(\Phi_{(i,j)}\) &\omdef=  \sum_{0<k < \frac{m}{2}} \Bigl(\{ (i,j),(0,k) \} + \{ (i,j),(k,\tfrac{m}{2})\}\Bigr)  
  + \{ (i,j),(0,\tfrac{m}{2})'\}  + \{ (i,j),(0,\tfrac{m}{2})''\} 
\\&=  \sum_{0<k < \frac{m}{2}} \(\tfrac{1-(-1)^j}{m} \(\xi^{-ik} + \xi^{ik}\) -   \tfrac{1-(-1)^i}{m} \(\xi^{jk} + \xi^{-jk}\)\) + 2\cdot\tfrac{(-1)^i-(-1)^j}{m}
\\&=\begin{cases} 2,&\text{if $i=0$ and $j$ is odd},\\0,&\text{otherwise},
\end{cases}
\ea\] and also, for  $x \in \{ (0,\tfrac{m}{2})', (0,\tfrac{m}{2})''\}$,
\[ \ba  \FT \fsd\(\Phi_{x}\) &\omdef= \sum_{0<k<\frac{m}{2}} \Bigl(\{ x,(0,k) \}  + \{ x,(k,\tfrac{m}{2}) \} \Bigr)+ \{ x,(0,\tfrac{m}{2})'\}  + \{ x,(0,\tfrac{m}{2})''\} 
\\&=
\tfrac{1-(-1)^{m/2}}{2} 
=
\begin{cases} 1,&\text{if $m\equiv 2\modu 4)$},\\0,&\text{if $m\equiv0\modu 4)$}.\end{cases}
\ea\]
 (Note that  we  need to check that 
 $\FT\fsd(\One) = \FT\fsd(\sgn) = 1$ as well, but this is obvious.)
Proving each of these three identities is straightforward.  Thus the funciton $\fsd$ satisfies conditions (1)-(3) in the theorem, and we deduce that it is the only such function by the positivity axiom of a fusion datum, exactly as in the classical case.


Suppose finally that $(W,S)$ has one of the remaining exceptional types. 
If $(W,S)$ is not of type $H_4$, then one can compute directly using the tables in \cite{Casselman} and Proposition \ref{h3-tbl} that
 the function $\fsd : \Uch(W)\to\RR$ which is 1 or 0 according to whether $\Eig(\Phi)$ is real or non-real satisfies (1)-(3) as well as (\ref{all-note}).  That $\fsd$ is the unique  function satisfying (1)-(3) then follows as in the previous cases 
by Theorem \ref{p4} and the positivity axiom of a fusion datum.
Similarly, if $(W,S)$ has type $H_4$, then one can calculate using Proposition \ref{h4-tbl} that (1)-(3) and (\ref{all-note}) hold for the function $\fsd :\Uch(W)\to\RR$ which is 0 on all $\Phi \in \Uch(W)$ with $\Eig(\Phi) \notin \RR$, $-1$ on the two elements of $\Uch(W)$ whose degrees have the form $\frac{1}{60} x^6 + \text{higher powers of $x$}$, and 1 on all other elements of $\Uch(W)$. Here, the uniqueness of $\fsd$ does not follow immediately from Theorem \ref{p4} because the values of $\fsd$ are not all positive, but it can still be easily checked.  In this case, the positivity axiom of a fusion datum implies that nearly all of the functions $\fsd' :\Uch(W) \to \RR$ satisfying (1) and (2) must have $\FT \fsd'(\Phi) < \FT\fsd(\Phi)$ for one of the special characters $\Phi\in \Irr(W)$. The remaining list of functions on $\Uch(W)$ which could possibly satisfy (1)-(3) is quite small,
and a short calculation confirms that the given function $\fsd$ is indeed  the only one with these properties.
%
%
%
\end{proof}

\section{Left cells and Kottwitz's conjecture}\label{lcells}
\def\Des{\mathrm{Des}_L}

In this final section we investigate how the preceding material connects to the left cells of a Coxeter group, and prove some partial results related to Conjecture \ref{kottwitz-conj} in the introduction.

Let $(W,S)$ be a finite Coxeter system with length function $\ell : W \to \NN$. We denote the left descent set of an element $w \in W$ by  $\Des(w) \omdef= \{ s\in S : \ell(sw) < \ell(w)\}$, and for each pair of elements $y,w \in W$, we write \[P_{y,w}(x) \in \NN[x]\] for   the associated \emph{Kazhdan-Lusztig polynomial}, as defined in  \cite[Chapter 5]{Bjorner} or \cite[\S12.5]{C} or \cite[\S7.9]{Humphreys}, among other places. Recall that $P_{w,w}(x) = 1$ and that $P_{y,w}(x) = 0$ unless $y \leq w$, where $\leq $ denotes the Bruhat partial order on $W$. 
From  \cite{KL}, we have the following sequence of definitions for $y,w \in W$:
\begin{enumerate}
\item[$\bullet$] Write $y \prec w$ if $y < w$ and $P_{y,w}(x)$ has degree $\frac{1}{2} \( \ell(w) - \ell(y) - 1\)$.
\item[$\bullet$] Write $y \leq_L w$ if there exist elements $x_0,x_1,\dots,x_k \in W$ such that $y = x_0$ and $w = x_k$ and for each $i \in [k]$, these two conditions hold: (1) either $x_{i-1} \prec x_i$ or $x_i \prec x_{i-1}$, and  (2) the descent set $\Des(x_{i-1}) $ is not contained in $\Des(x_i)$.


\item[$\bullet$] Write $y \sim_L w$ if $y \leq_L w$ and $w \leq_L y$.
\end{enumerate}
The \emph{left cells} of $W$ are  the equivalence classes of the relation $\sim_L$.
Let $V_\Gamma = \QQ\spanning \{ c_w : w \in \Gamma\}$ be a vector space with a basis indexed by a left cell $\Gamma$ in $W$, and define a map $\varrho_\Gamma : S \to \GL(V_\Gamma)$  by linearly extending  the formula
\[ \varrho_\Gamma(s) c_w = \begin{cases} -c_w, &\text{if }s \in \Des(w), \\
 \ds c_w +  \sum_{\substack{y \in \Gamma \\ s \in \Des(y)}} \mu(y,w) c_y,&\text{if }s \notin \Des(w),\end{cases}\] for $s \in S$ and $w \in \Gamma$,
 where $\mu(y,w)$ denotes the coefficient of $x^{\frac{1}{2}\( \ell(w)-\ell(y) -1\)}$ in $P_{y,w}(x)$ (which is zero unless $y \prec w$). This  extends to
a representation of $W$, called  the \emph{left cell representation} of $\Gamma$,  whose character we denote by
$\chi_\Gamma$.  

The following properties of  left cells and left cell representations are useful to recall (see \cite[Chapter 6]{Bjorner}). First, the left cell representations decompose the regular representation of $W$ and so the sum over left cells $\sum_\Gamma \chi_\Gamma$ is equal to $ \sum_{\psi \in \Irr(W)} \psi(1) \psi$.
Second, the singleton set $\{ 1\}$ is always a left cell, and its 
character is the trivial character $\One$. Finally, if $w_0$ denotes the longest element of $W$ and $\Gamma$ is a left cell, then $w_0\Gamma$ and $\Gamma w_0$ are left cells and $\chi_{w_0\Gamma} = \chi_{\Gamma w_0} = \chi_\Gamma \cdot \sgn$. 

We now state a few less well-known facts about the characters $\chi_\Gamma$.

\begin{theorem} If $\Gamma$, $\Gamma'$ are left cells in a finite Coxeter group $W$, then $\langle \chi_\Gamma,\chi_{\Gamma'} \rangle = |\Gamma \cap \Gamma'^{-1}|$.
\end{theorem}

\begin{proof}
When $W$ is a Weyl group, this is \cite[Proposition 12.15]{L}. Alvis notes how to extend Lusztig's proof to type $H_4$ \cite[Proposition 3.4]{Alvis}, and his argument remains valid in types $H_3$ and $I_2(m)$ (noting the explicit description of the left cells for these groups given in the following sections.) Alternatively, Geck has given a general proof of this theorem; see \cite[Corollary 3.9]{Geck_recent}.
\end{proof}

%

\begin{corollary}\label{lcell-cor} If $\Gamma$ is a left cell in a finite Coxeter group $W$, then $\chi_\Gamma$ is multiplicity-free if and only if every $w \in \Gamma \cap \Gamma^{-1}$ is an involution.
\end{corollary}

\begin{proof}
The character $\chi_\Gamma$ is multiplicity-free if and only if the inequality $\langle \chi_\Gamma,\chi_\Gamma\rangle \geq  \sum_{\psi \in \Irr(W)} \langle \chi_\Gamma,\psi\rangle$ is an equality. The left side is $|\Gamma \cap \Gamma^{-1}|$ by the previous theorem, while the right side is the number of involutions in $\Gamma$ by \cite[Theorem 1.1]{Geck_recent}.
\end{proof}

The next theorem shows the Fourier transform matrix of $\Uch(W)$ indeed satisfies property (P3) in the introduction. 

\begin{theorem}\label{ft-cells-thm} Let $(W,S)$ be a finite, irreducible Coxeter system with 
 Fourier transform matrix $\FT$, as  in Section \ref{FT-sect}. 
Fix a left cell $\Gamma$  of $W$, and let $v$ be the vector indexed by $\Uch(W)$ whose entries are the irreducible multiplicities of $\chi_\Gamma$, extended by zeros on $\Uch(W)\setminus \Irr(W)$. Then $\FT v = v$.
\end{theorem}

We  prove this here, even though to do so we must appeal to a few results not yet given.

\begin{proof}
When $W$ is a Weyl group, 
 this is  \cite[Theorem 12.2]{L}.
In types $H_3$ and $H_4$, the theorem follows from a  computer calculation using the description of the characters $\chi_\Gamma$ given in Sections \ref{h3lcells-sect} and \ref{h4lcells-sect} below. In type $I_2(m)$ it follows by a short argument using the content of Sections \ref{i2ft} and \ref{i2cells}
 that   the theorem  is equivalent to   the following identities: if $\xi = \exp\(\frac{2\pi \sqrt{-1}}{m}\)$ and $0<i<j<i+j < m$, then when $m$ is odd
\[ 
\sum_{0<k<\frac{m}{2}} \tfrac{1}{m} \( \xi^{-ik} + \xi^{ik} - \xi^{jk} - \xi^{-jk}\) = \begin{cases} 1,&\text{if $i=0$}, \\ 0,&\text{otherwise,}\end{cases}
 \]
 and when $m$ is even
\[\begin{cases} \ds \tfrac{(-1)^i -(-1)^j}{m} + \sum_{0<k<\frac{m}{2}} \tfrac{1}{m} \( \xi^{-ik} + \xi^{ik} - \xi^{jk} - \xi^{-jk}\) = \begin{cases} 1,&\text{if $i=0$}, \\ 0,&\text{otherwise,}\end{cases} 
\\[-10pt]\\
\ds\tfrac{1-(-1)^{m/2}}{2m} +  \sum_{0<k<\frac{m}{2}} \tfrac{1-(-1)^k}{m}  = \frac{1}{2}.
\end{cases}
\]
   Proving each of these is straightforward arithmetic.
\end{proof}

The last object of this section is to prove the following theorem. We will accomplish this in a case-by-case fashion, by examining the left cell representations in each non-crystallographic type.

\begin{theorem}\label{lcells-conj-thm} Conjecture \ref{kottwitz-conj} holds if $(W,S)$ has type $H_3$, $H_4$, or $I_2(m)$.
\end{theorem}

From this and the results of \cite{Casselman,Kottwitz}, it follows that Kottwitz's conjecture  holds for all finite irreducible Coxeter groups except possibly those of type $BC_n$, $D_n$, $E_7$, and $E_8$. 
Recent work of Bonnaf\'e and Geck \cite{Geck1,Geck2} has established the conjecture in all of these remaining cases except $E_8$.

%

\subsection{Left cells of the Coxeter group of type $H_3$}\label{h3lcells-sect}

Let $(W,S)$ be the Coxeter system of type $H_3$. The group $W$ decomposes into 22 distinct left cells which we may describe  as follows. (This description does not seem to appear anywhere in the literature, though one can easily compute the left cells in type $H_3$ directly, using for example Fokko du Cloux's program {\tt{Coxeter}} \cite{Coxeter}, which is what was used to derive the following statements.)
Label the generators $S = \{a,b,c\}$ as in Section \ref{h3-sect} and let $w_0 = (abc)^5$ denote the longest element of $W$.
Following Alvis \cite{Alvis}, we define \be \label{r_j}
R_J = \{ w \in W : \Des(w) = J\},\quad\text{for each subset $J\subset S$},\ee and  let $X^* =  \{ w_0w : w \in X\}$
for any subset $X \subset W$. 
We now define 12 subsets $I_i, J_i, K_i \subset W$ as follows:
\[ \ba I_1 & = R_{\{b,c\}} \cap R_{\{a\}} aba,\\ 
I_2 &= I_1a, \\
 I_3 &= I_2b=I_2^*,\\ 
 I_4 &= I_3a = I_1^*,\\ {}
\ea
\qquad
\ba 
J_1 &= R_{\{a,c\}} \cap R_{\{c\}} cbab, \\ 
J_2 &= J_1 b, \\
J_3 &= J_2a, \\
J_4 & = J_3b, \\
J_5 &= J_4c, 
\ea
\qquad 
\ba K_1 &= R_{\{a\}} - (I_4 \cup J_3),
\\
K_2 &= R_{\{b\}} - (I_3 \cup J_1^* \cup J_2 \cup J_4),
\\
K_3 &= R_{\{c\}} - J_5. \\ {} \\ {}
\ea
\]
The reader should compute that $|I_i| = 8$ and $|J_i| = 5$ and $|K_i| = 6$.
In addition, let 
$L = \{ 1\}$ so that $L^* = \{ w_0\}$.
We now have this computational proposition (which is closely related to the calculations summarized in \cite[\S5]{L_Benson}):

\begin{proposition}\label{h3-left} 
The left cells of the Coxeter system $(W,S)$ of type $H_3$ are the 22 disjoint subsets $I_i$, $J_i$, $J_i^*$, $K_i$, $K_i^*$, $L$,  $L^*$, and the characters of the associated cell representations are respectively 
\[(\phi_{4,3} + \phi_{4,4});\quad \phi_{5,2};\quad \phi_{5,5};\quad (\phi_{3,1} + \phi_{3,3});\quad (\phi_{3,6} + \phi_{3,8});\quad \phi_{1,0};\quad  \phi_{1,15}.\]
\end{proposition}

\begin{proof}
 {\tt{Coxeter}} \cite{Coxeter} outputs a description of the left cells  and the associated cell representations in type $H_3$ in a few seconds.
 We have checked that this information matches what is asserted in the proposition
using \sc{Magma} \cite{magma}.
\end{proof}


Table \ref{lcells-h3} lists (alongside some helpful auxiliary data) the sizes of the intersections of the left cells in $W$ with the group's four conjugacy classes of involutions, which we 
recall from Section \ref{h3-sect} are represented by the elements $1$, $a$, $ac$, and $(abc)^5$. 
 Each row of the table corresponds to a left cell, while each of the last four columns  corresponds to a conjugacy class.
Comparing Proposition \ref{h3-tbl} with Table \ref{lcells-h3} yields a proof of Conjecture \ref{kottwitz-conj} in type $H_3$ by inspection.

\begin{table}[b]
\[ \barr{l|l|l|llll}\hline
\text{Left cell}\quad & \text{Cell size}\quad & \text{Cell character}\quad & 
1 & (abc)^5 & a & ac
\\
\hline
I_i\ (1\leq i \leq 4) & 8 & \phi_{4,3} + \phi_{4,4} & 0 &0& 1  &1
\\
J_i\ (1\leq i \leq 5)  &5 &  \phi_{5,2} & 0 &0 &0 & 1
\\
J_i^*\ (1\leq i \leq 5)  &5 &  \phi_{5,5} & 0&0 & 1 & 0
\\
K_i\ (1\leq i \leq 3)  & 6 & \phi_{3,1} + \phi_{3,3} & 0&0 & 2 & 0
\\
K_i^*\ (1\leq i \leq 3)  &6 &  \phi_{3,6} + \phi_{3,8} & 0&0 & 0 & 2 
\\
L &1 &  \phi_{1,0}  & 1 & 0 & 0 &0
\\
L^* & 1 & \phi_{1,15}  & 0 & 1 & 0 & 0
\\
\hline\earr\] 
\caption{Sizes of intersections of left cells and  conjugacy classes of 
involutions   in type $H_3$
}
\label{lcells-h3}
\end{table}


\subsection{Left cells of the Coxeter group of type $H_4$}\label{h4lcells-sect}

The Coxeter group of type $H_4$ decomposes into a disjoint union of 206 left cells. Alvis classifies these in \cite{Alvis}, 
and assigns each of them one of the labels $A_i$, $B_i$, $B_i^*$, $C_i$, $C_i^*$, $D_i$, $D_i^*$, $E_i$, $E_i^*$, $F_i$, $F_i^*$, $G_1$, $G_1^*$.  We refer to \cite{Alvis} for the precise definition of these sets, noting the following correction:

\begin{remark}
On page 162 of the published version of 
Alvis's paper \cite{Alvis}, the left cell $A_{12}$ is defined by the equation \[\tag{INCORRECT} A_{12} = A_{10}d\] which  should  instead be \[\tag{CORRECT} A_{12} = A_{11}d.\] Apart from this quite minor (but inevitably confusing) detail,  everything else in Alvis's paper seems to be completely accurate.
\end{remark}


Table \ref{lcells-h4} lists  the sizes of the intersections of the left cells in $W$ with  the group's five conjugacy classes of involutions, which are represented as in Section \ref{h4-sect} by the elements $1$, $a$, $ac$, $(abc)^5$, and $(abcd)^{15}$. We have structured this table exactly like Table \ref{lcells-h3}, and comparing it to Proposition \ref{h4-tbl} similarly yields an immediate proof of 
Conjecture \ref{kottwitz-conj} in type $H_4$.

\begin{table}[b]
\[\barr{l|l|l|lllll}\hline
\text{Left cell}\quad & \text{Cell size}\quad & \text{Cell character}\quad & 
1 & (abcd)^{15} & a & (abc)^5 & ac
\\
\hline
A_i\ (1\leq i \leq 8) & 326 & \phi_{24,6} + \phi_{24,12} + \phi_{24,7}   & 0 &0& 0  &0 & 12
\\
&  &\ + \phi_{24,11}+ \phi_{8,12} + \phi_{8,13}  &
\\
 &  &\ + \phi_{18,10} + \phi_{30,10,12} + \phi_{30,10,14}  \\
 & & \ + \phi_{40,8} + 2 \phi_{48,9} 
 \\&&&&&&
\\[-8pt]
A_i\ (9\leq i \leq 18)  & 392 &  \phi_{24,6} + \phi_{24,12} + \phi_{24,7}  & 0 &0& 0  &0 & 14
\\
 &  &  \ + \phi_{24,11}  + \phi_{10,12} + \phi_{16,11}  \\
  &  &  \ + \phi_{16,13} + \phi_{18,10} + \phi_{30,10,12} \\
  & & \ + \phi_{30,10,14} + 2\phi_{40,8} + 2 \phi_{48,9} 
\\&&&&&&
\\[-8pt]
A_i\ (19\leq i \leq 24)  &436 & \phi_{24,6}  + \phi_{24,12} + \phi_{24,7} & 0 &0& 0  &0 & 16
\\
 & & \ + \phi_{24,11} + \phi_{6,12} + \phi_{6,20}  \\
 & & \ + \phi_{16,11} + \phi_{16,13}+ 2\phi_{30,10,12}  \\
 & & \ + 2\phi_{30,10,14} + 2\phi_{40,8}+2\phi_{48,9} 
\\&&&&&&
\\[-8pt]
B_i\ (1\leq i \leq 36) & 36 & \phi_{36,5} & 0 &0& 1  &0 & 0
\\
B_i^*\ (1\leq i \leq 36)  &36 & \phi_{36,15}  & 0 &0& 0  &1 & 0 
\\
C_i\ (1\leq i \leq 25) & 36 & \phi_{25,4} & 0 &0& 0  &0 & 1
\\
C_i^*\ (1\leq i \leq 25)  &36 & \phi_{25,16}  & 0 &0& 0  &0 & 1
\\
D_i\ (1\leq i \leq 16) & 32 & \phi_{16,3} + \phi_{16,6} & 0 &0& 1  &0 & 1
\\
D_i^*\ (1\leq i \leq 16)  &32 & \phi_{16,18} + \phi_{16,21}  & 0 &0& 0  &1 & 1
\\
E_i\ (1\leq i \leq 9) & 18 & \phi_{9,2} + \phi_{9,6} & 0 &0& 0  &0 & 2
\\
E_i^*\ (1\leq i \leq 9)  &18 & \phi_{9,22} + \phi_{9,26}  & 0 &0& 0  &0 & 2
\\
F_i\ (1\leq i \leq 4) & 8 & \phi_{4,1} + \phi_{4,7} & 0 &0& 2  &0 & 0
\\
F_i^*\ (1\leq i \leq 4)  & 8 & \phi_{4,31} + \phi_{4,37}  & 0 &0& 0  &2 & 0
\\
G_1& 1 & \phi_{1,0} & 1 &0& 0  &0 & 0
\\
G_1^*  &1 & \phi_{1,60}   & 0 &1& 0  &0 & 0 
\\
\hline\earr\] 
\caption{Sizes of intersections of  left cells and  conjugacy classes of 
involutions   in type $H_4$
}
\label{lcells-h4}
\end{table}

\subsection{Left cells of the Coxeter groups of type $I_2(m)$}
\label{i2cells}

Let $(W,S)$ be the Coxeter system of type $I_2(m)$ for $m\geq 3$, with $S = \{r,s\}$ and $w_0 \in W$ defined as in Section \ref{i2-sect}. The group $W$ then decomposes into a disjoint union of four left cells, given by the singleton sets $X = \{1\}$ and $X^* = \{w_0\}$ together with 
\[ 
 Y  = \{s,rs,srs,\dots, \underbrace{(\cdots srs)}_{m-1\text{ factors}}\} = R_{\{s\}}
 \qquad\text{and}\qquad
  Y^* = \{ r,sr,rsr,\dots, \underbrace{(\cdots rsr)}_{m-1\text{ factors}} \} = R_{\{r\}}
,\] 
with $R_J$  defined by (\ref{r_j}).
The following description of the left cell representations of $W$ can be found, for example, in \cite[Section 7.1]{Geck9}.

\begin{proposition}[See Geck \cite{Geck9}] 
The left cells of the Coxeter system $(W,S)$ of type $I_2(m)$ are the 4 disjoint subsets $X$, $X^*$, $Y$, $Y^*$, and the characters of the associated cell representations are respectively 
\[ \left\{ \barr{rrrrl}  \phi_{1,0},\quad &\phi_{1,m},\quad &\sum_{0<k<\frac{m}{2}} \phi_{2,k},\quad &\sum_{0<k<\frac{m}{2}} \phi_{2,k},\quad&\text{if $m$ is odd},\\
\phi_{1,0},\quad &\phi_{1,m},\quad &\phi_{1,m/2}' + \sum_{0<k<\frac{m}{2}} \phi_{2,k},\quad &\phi_{1,m/2}''+\sum_{0<k<\frac{m}{2}} \phi_{2,k},\quad&\text{if $m$ is even}.
\earr \right.\]
\end{proposition}

The description of the left cells is also noted in \cite[\S7.15]{Humphreys}. It is easy to compute the characters of the left cell representations directly, and we have included a short argument for completeness. 

\begin{proof}
We have $\chi_X = \phi_{1,0} = \One$ and $\chi_{X^*} = \phi_{1,m} = \sgn$ automatically, and the remaining  left cell characters are multiplicity-free by Corollary \ref{lcell-cor}. Since the left cell representations decompose the regular representation and since $\chi_{Y^*} = \chi_Y \cdot \sgn$, we must have $\chi_Y = \chi_{Y^*} = \sum_k \phi_{2,k}$ if $m$ is odd, and if $m$ is even, we must have either $\chi_Y = \phi'_{1,m/2} + \sum_k \phi_{2,k}$ and $\chi_{Y^*} = \phi''_{1,m/2} + \sum_k \phi_{2,k}$ or the reverse assignments. One resolves the ambiguity when $m$ is even by computing the values of the characters at  $r$ or $s$.\end{proof} 

Table \ref{lcells-i2} finally lists  the sizes of the intersections of the four left cells in $W$ with  the group's two or four conjugacy classes of involutions (see Section \ref{i2-sect}). We have structured this table exactly like Tables \ref{lcells-h3} and \ref{lcells-h4}, and comparing it to Table \ref{i2-tbl} completes the proof of  Theorem \ref{lcells-conj-thm}.

\begin{table}[b]
\[
\ba 
&m\text{ odd}
\\
& \barr{l|l|l|ll}\hline
\text{Left cell}\quad & \text{Cell size}\quad & \text{Cell character}\quad & 
1 & r
\\
\hline
X & 1 & \phi_{1,1} & 1 &0
\\
X^*  &1 &  \phi_{1,m}  & 0 &1
\\
Y  &m-1 &  \sum_{0<k<\frac{m}{2}} \phi_{2,k} & 0& (m-1)/2
\\
Y^* & m-1 & \sum_{0<k<\frac{m}{2}} \phi_{2,k} & 0 & (m-1)/2
\\
\hline\earr
\\
\\
&m\equiv 2 \modu 4)
\\
&\barr{l | l | l | llll}\hline
\text{Left cell}\quad & \text{Cell size}\quad & \text{Cell character}\quad & 
1 & w_0 & r & s
\\
\hline
X & 1 & \phi_{1,1} & 1 &0 & 0 & 0
\\
X^*  &1 &  \phi_{1,m}  & 0 &1 & 0 & 0
\\
Y  &m-1 & \phi'_{1,m/2} + \sum_{0<k<\frac{m}{2}} \phi_{2,k} & 0 & 0 & (m-2)/4 & (m+2)/4
\\
Y^* & m-1 &\phi_{1,m/2}''+ \sum_{0<k<\frac{m}{2}} \phi_{2,k} & 0 & 0 &  (m+2)/4 & (m-2)/4
\\
\hline\earr
\\
\\
&m \equiv 0 \modu 4)
\\
& \barr{l | l | l | llll}\hline
\text{Left cell}\quad & \text{Cell size}\quad & \text{Cell character}\quad & 
1 & w_0 & r & s
\\
\hline
X & 1 & \phi_{1,1} & 1 &0 & 0 & 0
\\
X^*  &1 &  \phi_{1,m}  & 0 &1 & 0 & 0
\\
Y  &m-1 & \phi'_{1,m/2} + \sum_{0<k<\frac{m}{2}} \phi_{2,k} & 0 & 0 & m/4 & m/4
\\
Y^* & m-1 &\phi_{1,m/2}''+ \sum_{0<k<\frac{m}{2}} \phi_{2,k} & 0 & 0 &  m/4 & m/4
\\
\hline\earr
\ea\] 
\caption{Sizes of intersections of  left cells and  conjugacy classes of 
involutions   in type $I_2(m)$
}
\label{lcells-i2}
\end{table}

\end{document}